\documentclass{amsart}
\usepackage{soul}
\usepackage{color}
\usepackage{amsmath}
\usepackage{xcolor,cancel}
\usepackage{array, booktabs, multicol}

\usepackage{kotex}
\usepackage{amssymb}
\usepackage{amsthm}
\usepackage{amscd}
\usepackage{bm}
\usepackage{soul}
\usepackage{eucal} 
\usepackage{enumerate}

\usepackage{shuffle}
\usepackage[utf8]{inputenc}
\usepackage[all]{xy}

\usepackage{kotex}
\allowdisplaybreaks

\usepackage{hyperref}
\hypersetup{
  colorlinks   = true,          
  urlcolor     = blue,          
  linkcolor    = blue,          
  citecolor   = red             
}
\usepackage{cleveref}

\theoremstyle{plain}
\newtheorem{theorem}{Theorem}[section]
\newtheorem{conjecture}[theorem]{Conjecture}
\newtheorem{lemma}[theorem]{Lemma}
\newtheorem{proposition}[theorem]{Proposition}
\newtheorem{corollary}[theorem]{Corollary}

\newtheorem{theoremx}{Theorem}

\theoremstyle{definition}
\newtheorem{definition}[theorem]{Definition}

\numberwithin{equation}{section}

\newcommand{\ppar}{${}$\par}

\newcommand\fantome[1]{}

\def\bN{\mathbb N}

\def\fa{\mathfrak{a}}
\def\fb{\mathfrak{b}}

\def\fu{\mathfrak{u}}

\def\Fq{\mathbb F_q}

\newcommand{\depthlex}{(\depth, \lex)}

\DeclareMathOperator{\Si}{Si}
\DeclareMathOperator{\Ind}{Ind}

\DeclareMathOperator{\Init}{Init}

\DeclareMathOperator{\Li}{Li}

\newcommand{\F}{\mathbb{F}}
\newcommand{\C}{\mathbb{C}}

\newcommand{\fs}{\mathfrak{s}}
\newcommand{\ft}{\mathfrak{t}}
\newcommand{\fn}{\mathfrak{n}}
\newcommand{\fm}{\mathfrak{m}}
\newcommand{\fp}{\mathfrak{p}}

\newcommand{\N}{\ensuremath \mathbb{N}}

\DeclareMathOperator{\depth}{depth}
\DeclareMathOperator{\lex}{lex}

\author[B.-H. Im]{Bo-Hae Im}
\address{
Dept. of Mathematical Sciences, KAIST,
291 Daehak-ro, Yuseong-gu,
Daejeon 34141, South Korea
}
\email{bhim@kaist.ac.kr}

\author[H. Kim]{Hojin Kim}
\address{
Normandie Université,
Université de Caen Normandie - CNRS,
Laboratoire de Mathématiques Nicolas Oresme (LMNO), UMR 6139,
14000 Caen, France.
}
\email{hojin.kim@unicaen.fr}

\author[T. Ngo Dac]{Tuan Ngo Dac}
\address{
Normandie Université,
Université de Caen Normandie - CNRS,
Laboratoire de Mathématiques Nicolas Oresme (LMNO), UMR 6139,
14000 Caen, France.
}
\email{tuan.ngodac@unicaen.fr}

\title[The threshold for linear independence of MZV's]{The threshold for linear independence of multiple zeta values in positive characteristic}



\date{\today}

\begin{document}


\begin{abstract}
A fundamental conjecture formulated by Thakur in 2009, which has guided significant developments in function field arithmetic, asserts that multiple zeta values (MZV's) in positive characteristic of fixed weight are linearly independent over $\Fq$​. 

In this paper we settle this conjecture by determining the precise threshold for this independence. We prove that linear independence holds for all weights up to $2q$, while for weight $2q+1$ we establish the existence of a unique and explicit $\Fq$-linear relation. This result provides the first counterexample to Thakur’s conjecture. Our proof relies on a new connection between MZV's and Carlitz multiple polylogarithms over $\Fq$, generalizing a central result of \cite{IKLNDP24}. We also introduce a modification of the algorithm from \cite{ND21} that yields a weight-preserving operator acting on $\Fq$-linear relations, providing the algebraic framework for these results.
\end{abstract}

\maketitle
\tableofcontents


\section{Introduction} \label{sec: introduction}

\subsection{Euler's multiple zeta values} \ppar

Multiple zeta values (MZV's) introduced by Euler are real positive numbers defined by the series
\[ \zeta(n_1,\dots,n_r):=\sum_{0<k_1<\dots<k_r} \frac{1}{k_1^{n_1} \dots k_r^{n_r}}, \quad \text{where } n_i \geq 1, n_r \geq 2. \]
Here $r$ denotes the depth and $w=n_1+\dots+n_r$ denotes the weight of the presentation $\zeta(n_1,\dots,n_r)$. These values generalize the special values $\zeta(n)$ for $n \geq 2$ of the Riemann zeta function. In the 1990's Zagier~\cite{Zag94} revitalized the study of MZV's, initiating a period of intensive research across many mathematical fields including arithmetic geometry, number theory, $K$-theory and knot theory (see, e.g., \cite{Bro12,Del10,DG05,Dri90,GKZ06,Hof97,IKZ06,Iha89,Iha91,LM96,Ter02,Ter06}). At the same time, MZV's emerged in various branches of theoretical physics, such as quantum field theory, deformation quantization and high-energy physics (see, e.g., \cite{BPP20,Bro96a,BK97,Bro09,CDPP24,Kon99,Kon03}).

As detailed in \cite{BGF} MZV's have a rich algebraic structure and are equipped with stuffle and shuffle relations. A notable example is the identity discovered by Euler:
$$\zeta(3)=\zeta(1,2).$$

A fundamental objective in this theory is to characterize all $\mathbb Q$-linear and algebraic relations among MZV's. In \cite{IKZ06} Ihara, Kaneko, and Zagier conjectured that all linear relations between MZV's arise solely from the interaction of stuffle and shuffle relations. Precise conjectures formulated by Zagier \cite{Zag94} and Hoffman \cite{Hof97} predict the dimension and an explicit basis for the $\mathbb Q$-vector space spanned by MZV's of fixed weight. Despite significant breakthroughs in recent decades by Terasoma \cite{Ter02}, Deligne-Goncharov \cite{DG05} and Brown \cite{Bro12}, we are very far from the complete resolution of this problem. We refer the reader to \cite{BGF, Del13, Zag94} for a comprehensive exposition.

\subsection{Multiple zeta values in positive characteristic} \ppar

Following the well-established analogy between number fields and function fields (see \cite{Iwa69,MW83,Wei39}), we consider the setting of function fields in positive characteristic. Let $A=\Fq[\theta]$ be the polynomial ring in the variable $\theta$ over a finite field $\Fq$ of $q$ elements and characteristic $p>0$, and let $A_+$ denote the set of monic polynomials in $A$. Let $K=\Fq(\theta)$ be the fraction field of $A$ equipped with the rational point~$\infty$. We denote by $K_\infty=\Fq(\!(1/\theta)\!)$ the completion of $K$ at $\infty$ and by $\C_\infty$ the completion of a fixed algebraic closure $\overline K$ of $K$ at $\infty$.

In 1935 Carlitz \cite{Car35} initiated the analogue of the theory by introducing the zeta values $\zeta_A(n)$ for $n \in \N$ defined as
\[ \zeta_A(n) := \sum_{a \in A_+} \frac{1}{a^n} \in K_\infty. \]
These values are considered as the function field analogues of classical Riemann zeta values. Building upon Carlitz's seminal work, for any tuple of positive integers $\mathfrak s=(s_1,\ldots,s_r) \in \N^r$, Thakur \cite{Tha04} defined the characteristic $p$ multiple zeta value (MZV) $\zeta_A(\fs)$ by
\begin{equation*}
\zeta_A(\fs):=\sum \frac{1}{a_1^{s_1} \dots a_r^{s_r}} \in K_\infty,
\end{equation*}
where the sum runs over tuples $(a_1,\ldots,a_r) \in A_+^r$ satisfying $\deg a_1 > \dots > \deg a_r$. The integers $r$ and $w(\fs)=s_1+\dots+s_r$ are called the depth and weight of the presentation $\zeta_A(\fs)$, respectively. We note that each MZV does not vanish (see \cite{Tha09b}) and that Carlitz zeta values are exactly depth one MZV's. Since their introduction, MZV's in positive characteristic have occupied a central position in function field arithmetic and have been the subject of intensive investigation (see \cite{AT90,AT09,CPY19,GP21,Gos96,IKLNDP23,IKLNDP24b,NDNCP26,Pap08,Pel12,Tha04,Tha09,Tha10, Tha17,Tha20,Wad41,Yu91}).

As mentioned by Thakur \cite{Tha04}, in the function field setting there are two natural analogues of $\mathbb{Q}$: the finite field $\Fq$ and the function field $K$. Thakur \cite{Tha09} also discovered the so-called fundamental relation which is  the function field counterpart to Euler's celebrated identity:
\[
\zeta_A(q)+D_1 \zeta_A(1,q-1)=0,
\]
where $D_1=\theta^q-\theta \in K$. Thakur subsequently formulated {\it two fundamental conjectures} regarding the linear relations among MZV's over both $\Fq$ and $K$, highlighting a dichotomy between linear independence over $\Fq$ and linear dependence over $K$, and some contrast with the classical theory. These conjectures shape the field by providing the key principles behind nearly all major results as detailed below.

\subsection{Linear independence over $\Fq$} \ppar

The first fundamental conjecture proposed by Thakur~\cite{Tha09} asserts the 
linear independence of MZV's over the finite field $\Fq$.

\begin{conjecture}[Thakur 2009]
\label{conj: independence over Fq}
Multiple zeta values in positive characteristic are linearly independent over $\Fq$.
\end{conjecture}

We refer the reader to~\cite{Tha09,Tha17} for an extensive discussion, including 
heuristic reasons based on the motivic interpretation developed in~\cite{And86,AT90,AT09} 
and the first evidence supporting this conjecture. In~\cite{Tha17} further numerical 
evidence is provided: an abundance of explicit $K$-linear relations among MZV's is 
presented, yet none of them is an $\Fq$-linear relation. We note that it suffices to prove the linear independence for MZV's of a fixed weight by the work of Chang \cite{Cha14}.

The prevailing intuition suggests that the function field setting is inherently more 
rigid than its classical counterpart, leading to a profound contrast with the classical 
theory of MZV's. As noted by Thakur~\cite{Tha09,Tha17}, the validity of this conjecture 
would carry several far-reaching structural implications 
(see \cite[\S 2]{ND21}, \cite[\S 5]{Tha09}, \cite[\S 6]{Tha17}):

\begin{enumerate}
    \item It would establish the principle of ``only one shuffle''.
    \item It would imply the associativity of the shuffle algebra of MZV's.
    \item It would ensure that the depth filtration on the space of MZV's is well-defined.
\end{enumerate}

This conjecture was long expected to be true and has served as a driving force 
for many significant developments in function field arithmetic. It has shaped 
the development of the theory of MZV's in positive characteristic, motivating 
foundational work on their algebraic structure, their motivic relations to 
periods and their connections to Carlitz multiple polylogarithms. In this 
sense the conjecture has been as valuable for the mathematics it inspired as 
for the question it posed. The study of linear independence over $\Fq$ is 
moreover intimately connected to the notion of universality as developed in 
the recent work of Pellarin~\cite{Pel25} (see also~\cite{CNDP23}).

\subsection{Linear dependence over $K$} \ppar \label{sec: independence over K}

By contrast, the situation changes completely for linear relations among MZV's over 
$K$ as there are plenty of them. We refer the reader to~\cite{Tha17} for more details.

Similar to the classical setting, a central objective is to characterize all $K$-linear 
relations among MZV's. However, due to the previously mentioned ``only one shuffle'' 
principle, determining a conjectural framework for these relations remained a significant 
challenge. It was only in 2018, following the work of Todd~\cite{Tod18}, that Todd and 
Thakur~\cite[\S 8]{Tha17} formulated the second fundamental conjecture for the dimension 
and basis of the $K$-span of MZV's of fixed weight. This is now considered the function 
field analogue of the Zagier-Hoffman conjectures.

\begin{conjecture}[Todd-Thakur 2018]
\label{conj: Zagier Hoffman} \ppar

\begin{enumerate}[\normalfont (1)]
\item \textup{(Zagier's conjecture in positive characteristic)} Let $d(w)$ be defined by
\begin{align*}
d(w) = \begin{cases}
1 & \text{if } w = 0, \\
2^{w-1} & \text{if } 1 \leq w \leq q-1, \\
2^{w-1} - 1 & \text{if } w = q,
\end{cases}
\end{align*}
and $d(w) = \sum_{i=1}^{q} d(w-i)$ for $w > q$.

For any $w \in \N$, let $\mathcal{Z}_w$ denote the $K$-span of MZV's of weight $w$. Then
\[ \dim_K \mathcal{Z}_w = d(w). \]

\item \textup{(Hoffman's conjecture in positive characteristic)}
A $K$-basis for $\mathcal{Z}_w$ is given by the set $\mathcal{T}_w$ consisting of 
MZV's $\zeta_A(s_1, \ldots, s_r)$ of weight $w$ such that $s_i \leq q$ for 
$1 \leq i < r$ and $s_r < q$.
\end{enumerate}
\end{conjecture}

In~\cite{ND21} the algebraic part of these conjectures was established 
(see \cite[Theorem~A]{ND21}): the third author proved that $\dim_K \mathcal{Z}_w \leq d(w)$ 
for all $w \in \N$. The method relies on the fundamental relation and Todd's operators. 
Furthermore, an algorithm was developed in~\cite{ND21} to express any MZV as a $K$-linear 
combination of elements of $\mathcal{T}_w$ relative to a suitable ordering. For the 
transcendental part, the author applied the Anderson-Brownawell-Papanikolas (ABP) 
criterion devised in~\cite{ABP04} to prove sharp lower bounds for weights $w \leq 2q-2$ 
(see~\cite[Theorem~D]{ND21}). However, extending this method to general weights presented 
substantial difficulties.

Subsequently, in~\cite{IKLNDP24} the authors and their colleagues succeeded in proving 
the transcendental part of Conjecture~\ref{conj: Zagier Hoffman}, namely 
$\dim_K \mathcal{Z}_w \geq d(w)$ for all $w \in \N$. To do so, we proved a variant of 
the aforementioned algorithm and uncovered a key connection between MZV's and Carlitz 
multiple polylogarithms (CMPL's) over $K$ (see~\S\ref{sec: CMPL} for a precise definition of 
CMPL's).
By applying the ABP criterion to the latter class of objects, we successfully resolved 
Conjecture~\ref{conj: Zagier Hoffman} in full generality (see also \cite{CCM23} where the same result was proved using the same method, objects and proof strategy but the presentation is different).

We mention that a characteristic zero analogue of 
the aforementioned algorithm was recently investigated in~\cite{HMSW25}. 

\subsection{Main results and ingredients of the proofs} \ppar

In this paper we settle the remaining Conjecture~\ref{conj: independence over Fq}. While our initial objective was to provide further evidence for this conjecture using the framework developed in~\cite{ND21,IKLNDP24}, we instead identify an explicit $\Fq$​-linear relation consisting of $12$ terms of weight $2q+1$. This relation, whose existence is invisible without the new framework developed here, provides a counterexample to Conjecture~\ref{conj: independence over Fq}. This result reinforces our belief that Thakur's MZV's are the correct analogues of Euler's MZV's: they share similar deep structures even when the methods required to prove them follow fundamentally different pathways.

Our first main result is as follows.

\begin{theoremx} \label{thm: weight 2q+1}
There exists a unique and explicit $\Fq$-linear relation among MZV's of weight $2q+1$. 
More precisely,
\begin{align*}
  &\zeta_A(q+2, q-1) + 2\zeta_A(3, 2q-2) + \zeta_A(q, 2, q-1) - \zeta_A(1, q, q) \\
  &- \zeta_A(1, 1, 2q-1) - \zeta_A(1, q-1, 1, q) + \zeta_A(q+1, 1, q-1) 
  + \zeta_A(2, q-1, 1, q-1) \\
  &+ \zeta_A(2, q, q-1) + \zeta_A(2, 1, 2q-2) + \zeta_A(q, 1, 1, q-1) 
  - \zeta_A(1, 1, q-1, q) = 0.
\end{align*}
\end{theoremx}

This result is optimal in the sense that $2q+1$ is the minimal weight for which an 
$\Fq$-linear relation exists. Indeed, we establish the following result:

\begin{theoremx} \label{thm: weight up to 2q}
The MZV's of weight $w$ are linearly independent over $\Fq$ for all $w < 2q+1$.
\end{theoremx}

We briefly comment on the methods used. As noted by Thakur~\cite{Tha09,Tha17}, 
Conjecture~\ref{conj: independence over Fq} is fundamentally different in nature from 
Conjecture~\ref{conj: Zagier Hoffman}, i.e., linear independence versus linear dependence. He provided heuristic arguments based on the 
motivic interpretation developed in~\cite{And86,AT90,AT09} to support the linear 
independence of MZV's over $\Fq$. For linear independence over $K$, the Anderson--Brownawell--Papanikolas 
criterion~\cite{ABP04} is a powerful and fruitful tool. Unfortunately, no analogue of 
this criterion exists for linear independence over $\Fq$.

Our strategy involves refining the following approach: first, applying the algorithm 
introduced in~\cite{ND21} to express an $\Fq$-linear relation among MZV's in the basis 
$\mathcal{T}_w$ (as defined in Conjecture~\ref{conj: Zagier Hoffman}) so as to obtain 
a $K$-linear relation, and second, showing that this relation is non-trivial. However, 
controlling the coefficients in the algorithm from~\cite{ND21} with sufficient precision 
is made impossible by the combinatorial complexity of the MZV shuffle relations.

To circumvent these difficulties, we introduce several key new ideas and ingredients. The first is a 
generalization of the connection between MZV's and CMPL's over $\Fq$ (rather than over 
$K$ as in \cite{IKLNDP24}). The method in \cite{IKLNDP24} which is based on that of \cite{ND21} does not extend to $\Fq$.
Our proof is direct and completely avoids the 
methods of~\cite{ND21,IKLNDP24}. This allows us to reduce the problem to the study of $\Fq$-linear 
relations among CMPL's for which the shuffle relations are more tractable.

\begin{theoremx} \label{thmx: new connection}
For any $w \in \N$, the $\Fq$-vector spaces spanned by MZV's and CMPL's of weight $w$ 
coincide.
\end{theoremx}

We further introduce two new ingredients: a modification of the algorithm from~\cite{ND21} 
and a new weight-preserving operator $\Ind$ acting on $\Fq$-linear relations.

Combining all these ingredients, we successfully carry out the strategy for CMPL's, 
leading to the counterexample and the linear independence for weights $w \le 2q$ presented in Theorems~\ref{thm: weight 2q+1} 
and~\ref{thm: weight up to 2q}, respectively.

\subsection*{Addendum (September 2026): beyond the threshold} \label{sec: addendum} \ppar

In the appendix we present our subsequent work beyond the threshold $w=2q+1$. 

First, in \S \ref{sec: BCP relations} we construct a family of $\Fq$-linear relations among CMPL's called primitive relations (or $P$ relations for short). This family $P(\fs)$ is indexed by nonempty tuples $\fs$ and the relation discovered in Theorem \ref{thm: weight 2q+1} is derived from $P(1)$ (i.e., $\fs=(1)$), see Prop. \ref{prop: P1}.

Second, applying Todd's operators $\mathcal B^*$ and $\mathcal C$, we obtain many relations called BCP relations. We show in Theorem \ref{thm: relations from P, B, C} that when $w<3q$, they form a basis for all $\Fq$-linear relations of CMPL's of weight $w$, thus also a basis for all $\Fq$-linear relations of MZV's of weight $w$ by Theorem \ref{thmx: new connection}. This extends Theorems \ref{thm: weight 2q+1} and \ref{thm: weight up to 2q}.

Finally, we speculate that for arbitrary weights, the vector space of all $\Fq$-linear relations of MZV's of weight $w$ is generated by BCP relations and write down the expected dimension of this vector space in \S \ref{sec: speculations}.

At the end of August 2026, we were in the process of working towards the general weights when we became aware of the paper \cite{HHLL26} (arXiv:2608.26636v1) in which the authors have been working on the same problem independently. 

We would like to say some words about our appendix and the work of \cite{HHLL26} (we thank the authors of \cite{HHLL26} for explanations of their work and helpful discussion):
\begin{enumerate}
\item The work \cite{HHLL26} started from and builds upon the first version of this article posted in arXiv (arXiv:2604.25451v1) in April 2026 as the authors of \cite{HHLL26} informed us. In this paper they construct the so-called quadruple-carry relations. Then with AI assistance they prove that the vector space of all $\Fq$-linear relations of MZV's of weight $w$ is generated by quadruple-carry relations and completely determined the dimension of this vector space. 

\item Compared to our work, we believe that BCP relations mentioned above are closely related to quadruple-carry relations and that the vector space spanned by BCP relations and that spanned by quadruple-carry relations coincide. The latter is supported by the fact that these vector spaces would have the same dimension by \cite{HHLL26} and \S \ref{sec: speculations}.

\item Our proofs and constructions were obtained entirely without AI assistance. 
\end{enumerate}

\subsection*{Plan of the paper} \ppar

The mathematical content and proofs of this paper, i.e., the definitions, the constructions, the statements and the proofs were obtained entirely without AI assistance.

The paper is organized as follows. In Section~\ref{sec: Brown} we review the definitions of MZV's and CMPL's, 
as well as analogues of Brown's theorem. Sections~\ref{sec: bridge} and~\ref{sec: algorithm revisited} detail the 
connection between MZV's and CMPL's over $\Fq$ (Theorem~\ref{thmx: new connection}), the modified algorithm and the 
new operator $\Ind$. Section~\ref{sec: small weights} provides a key result for weights 
$w \leq 2q+1$ (see Theorem~\ref{thm: small weights}). Finally, Section~\ref{sec: main results} uses Theorem~\ref{thm: small weights} and Theorem~\ref{thmx: new connection}
and provides the proofs of Theorems~\ref{thm: weight 2q+1} 
and~\ref{thm: weight up to 2q}.

In the appendix we present our subsequent work beyond the threshold $w=2q+1$. In \S \ref{sec: BCP relations} we construct $\Fq$-linear relations called BCP relations and use them to get all the $\Fq$-linear relations of weight $w<3q$ (see Theorem \ref{thm: relations from P, B, C}), extending Theorems \ref{thm: weight 2q+1} and \ref{thm: weight up to 2q}. We also present speculations for general $w$.

\subsection*{Acknowledgments} 
The third author (T. ND.) would like to express his gratitude to Federico Pellarin for continuous support and encouragement. 

The first author (B.-H. Im) was supported by the Basic Science Research Program through the National Research Foundation of Korea (NRF) grant funded by the Korea government (MSIT)(NRF-2023R1A2C1002385, or RS-2023-NR076333). H. Kim (H. K.) and T. Ngo Dac (T. ND.) were partially supported by the Excellence Research Chair FLCARPA funded by the Normandy Region. T. ND. was partially supported by the ANR grant PPAL ANR-25-CE40-4664.

Finally, all the authors thank the International Joint Lab LMNO-KAIST for support during this project.


\section{Analogues of Brown's theorem} \label{sec: Brown}

In \cite{ND21} one of the authors studies Zagier-Hoffman's conjectures for MZV's in positive characteristic formulated by Todd and Thakur (see Conjecture \ref{conj: Zagier Hoffman}). These conjectures predict the dimension and an explicit basis for the vector space $\mathcal Z_w$ over $K$ spanned by MZV's of fixed weight $w$. One of the main results in \cite{ND21} is an analogue of Brown's theorem in \cite{Bro12} which states that the vector space $\mathcal Z_w$ is generated by the conjectural basis (see \cite[Theorem A]{ND21}). This proves half of Zagier-Hoffman's conjectures in positive characteristic.

Based on \cite{ND21},  the other half of these conjectures was settled in \cite{IKLNDP24} (and independently in \cite{CCM23}). The key insight is a connection between MZV's and CMPL's; the proof of which relies heavily on the aforementioned analogue of Brown's theorem and its proof.

The main goal of this section is to review the notion of MZV's, that of CMPL's and then the previous connection between MZV's and CMPL's (see Theorem \ref{thm: connection}).

\subsection{Multiple zeta values (MZV's) in positive characteristic} \ppar

We recall the notation introduced in \S \ref{sec: introduction}. Let $q$ be a power of prime $p$ and $\Fq$ be a finite field of order $q$. Let $A=\Fq[\theta]$ be the polynomial ring in $\theta$ over $\Fq$ and $A_+$ the set of monic polynomials in $A$. We denote by $K=\Fq(\theta)$ the fraction field of $A$ equipped with the rational point~$\infty$. Let $K_\infty=\Fq(\!(1/\theta)\!)$ be the completion of $K$ at~$\infty$ and $\C_\infty$ be the completion of a fixed algebraic closure $\overline K$ of $K$ at $\infty$. Further, let $v_\infty$ be the discrete valuation on $K$ associated to the place $\infty$ normalized as $v_\infty(\theta) = -1$ and $|\cdot|_\infty = q^{-v_\infty(\cdot)}$ be the corresponding norm on $K$.

Let $\fs=(s_1,\dots,s_n)$ be a tuple of positive integers. Then $w(\fs)=s_1+\dots+s_n$ (resp. $n$) is called the weight (resp. the depth) of $\fs$. For a nonempty tuple $\fs$, we put $\fs_-:=(s_2,\dots,s_n)$.

For $d \in \mathbb Z$ we recall the definition of power sums given as in \cite[\S 2]{ND21}
\begin{equation*}
S_d(\fs)=\sum \frac{1}{a_1^{s_1} \ldots a_r^{s_r}} \in K_\infty
\end{equation*}
where the sum runs through the set of tuples $(a_1,\ldots,a_r) \in A_+^r$ with $d=\deg a_1>\ldots>\deg a_r$. Further, we define
\begin{equation*}
S_{<d}(\fs)=\sum \frac{1}{a_1^{s_1} \ldots a_r^{s_r}} \in K_\infty
\end{equation*}
where the sum is over $(a_1,\ldots,a_r) \in A_+^r$ with $d>\deg a_1>\ldots>\deg a_r$. Thus
\begin{align*}
& S_{<d}(\fs) =\sum_{i=0}^{d-1} S_i(\fs), \\
& S_{d}(\fs) =S_d(s_1) S_{<d}(\fs_-)=S_d(s_1) S_{<d}(s_2,\dots,s_r).
\end{align*}
Here, by convention we define empty sums to be $0$ and empty products to be $1$. In particular, $S_{<d}$ of the empty tuple is equal to $1$.

We recall the multiple zeta value (MZV) introduced by Thakur \cite{Tha04} as follows:
\begin{equation*}
    \zeta_A(\fs)  = \sum \limits_{d \geq 0} S_d(\fs) = \sum\limits_{\substack{a_1, \dots, a_n \in A_{+} \\ \deg a_1> \dots > \deg a_n\geq 0}} \dfrac{1}{a_1^{s_1} \dots a_n^{s_n}}  \in K_{\infty}.
\end{equation*}
We refer the reader to \cite{Tha17,Tha20} for excellent surveys of these objects.

\subsection{Brown's theorem for MZV's} \ppar

We also recall Chen's formula (see \cite{Che15}): for $s, t \in \mathbb{N}$, we have
\begin{equation} \label{eq: sum MZV}
S_d(s) S_d(t)  = S_d(s+t)  + \sum \limits_i \Delta^i_{s,t} S_d(s+t-i, i),
\end{equation}
where
\begin{equation} \label{eq:Delta Chen}
    \Delta^i_{s,t} = \begin{cases}
			(-1)^{s-1} {i - 1  \choose s - 1} +  (-1)^{t-1} {i-1 \choose t-1} & \quad  \text{if } q - 1 \mid i \text{ and } 0 < i < s + t, \\
            0 & \quad \text{otherwise.}
		 \end{cases}
\end{equation}

It follows that (see \cite{ND21}):
\begin{proposition} \label{prop: product MZV}
Let $\fs,\ft$ be two tuples. Then
\begin{enumerate}[\normalfont (1)]
    \item There exist $f_i \in \mathbb{F}_q$ and tuples $\mathfrak{u}_i$ with $\depth(\mathfrak{u}_i) \leq \depth(\fs) + \depth(\mathfrak{t})$ for all~$i$  such that
    \begin{equation*}
        S_d(\fs) S_d(\ft)  = \sum \limits_i f_i S_d(\mathfrak{u}_i)  \quad \text{for all } d \in \mathbb{Z}.
    \end{equation*}
    \item There exist $f'_i \in \mathbb{F}_q$ and tuples $\mathfrak{u}'_i$ with $\depth(\mathfrak{u}'_i) \leq \depth(\fs) + \depth(\mathfrak{t})$ for all~$i$  such that
    \begin{equation*}
        S_{<d}(\fs) S_{<d}(\ft) = \sum \limits_i f'_i S_{<d}(\mathfrak{u}'_i) \quad \text{for all } d \in \mathbb{Z}.
    \end{equation*}
    \item There exist $f''_i \in \mathbb{F}_q$ and tuples $\mathfrak{u}''_i$ with  $\depth(\mathfrak{u}''_i) \leq \depth(\fs) + \depth(\mathfrak{t})$ for all $i$  such that
    \begin{equation*}
        S_d(\fs) S_{<d}(\ft) = \sum \limits_i f''_i S_d(\mathfrak{u}''_i)  \quad \text{for all } d \in \mathbb{Z}.
    \end{equation*}
\end{enumerate}
\end{proposition}

We denote by $\mathcal{Z}$ the $K$-vector space generated by the MZV's and $\mathcal{Z}_w$ the $K$-vector space generated by the MZV's of weight $w$. It follows from Proposition~\ref{prop: product MZV} that $\mathcal{Z}$ is a $K$-algebra.

The fundamental relation $R_1$ is given by
\begin{equation}\label{eq:R1}
S_d(q)+D_1 S_{d+1}(1,q-1)=0,
\end{equation}
where $D_1=\theta^q-\theta \in K$. This is the analogue of the celebrated relation discovered by Euler: $\zeta(3)=\zeta(1,2)$.

As a direct consequence, we obtain a $K$-linear relation among MZV's,
\[
\zeta_A(q)+D_1 \zeta_A(1,q-1)=0.
\]
In fact, it is proved in \cite{ND21} that this is the $K$-linear relation of smallest weight among MZV's (see Theorem \ref{thm: generating set} below).

\begin{definition} \label{defn: basis MZV}
Let $w \in \N$.
\begin{enumerate}[\normalfont (1)]
\item We define $\mathcal J_w$ to be the set of tuples $\fs=(s_1,\ldots,s_r)$ of weight $w$ with $1\leq s_i\leq q$ for $1\leq i\leq r-1$ and $s_r<q$.
\item We also define $\mathcal J=\cup_{w\ge1} \mathcal J_w$.
\end{enumerate}
\end{definition}

The main results of \cite{ND21} read as follows:
\begin{theorem} \label{thm: generating set}
For $w \in \N$, we denote by $\mathcal T_w$ the set of MZV's $\zeta_A(\fs)$ with $\fs \in \mathcal J_w$. Then
\begin{enumerate}[\normalfont (1)]
\item The $K$-vector space $\mathcal Z_w$ is generated by $\mathcal T_w$.
\item  If $w \leq 2q-2$, the elements in $\mathcal T_w$ are linearly independent over $K$. Therefore, a $K$-basis for $\mathcal Z_w$ is given by  $\mathcal T_w$.
\end{enumerate}
\end{theorem}

Theorem \ref{thm: generating set} (1) could be considered as an analogue of Brown's theorem proved in \cite{Bro12} in our context. The proof is based on an algorithm which express an arbitrary MZV as a linear combination over $K$ of MZV's in the set $\mathcal T_w$ (see \S \ref{sec: algorithm revisited}).

\subsection{Carlitz multiple polylogarithms (CMPL's)} \ppar \label{sec: CMPL}

We now review the notion of multiple polylogarithms in positive characteristic (or Carlitz multiple polylogarithms). We put $\ell_0 = 1$ and $\ell_d = \prod^d_{i=1}(\theta - \theta^{q^i})$ for all $d \in \mathbb{N}$. For $\mathfrak s = (s_1 , \dots, s_n) \in \mathbb{N}^n$ and $d \in \mathbb Z$ we introduce analogues of power sums
\begin{align*}
\Si_d(\fs) &=\sum_{d=i_1>\ldots>i_r \geq 0} \frac{1}{\ell_{i_1}^{s_1} \ldots \ell_{i_r}^{s_r}} \in K_\infty, \\
\Si_{<d}(\fs) &=\sum_{d>i_1>\ldots>i_r \geq 0} \frac{1}{\ell_{i_1}^{s_1} \ldots \ell_{i_r}^{s_r}} \in K_\infty.
\end{align*}
Thus
\begin{align*}
& \Si_{<d}(\fs) =\sum_{i=0}^{d-1} \Si_i(\fs), \\
& \Si_{d}(\fs) =\Si_d(s_1) \Si_{<d}(\fs_-)=\Si_d(s_1) \Si_{<d}(s_2,\dots,s_r).
\end{align*}

We introduce the Carlitz multiple polylogarithm (CMPL for short) by
\begin{equation*}
    \Li(\fs) := \sum \limits_{d \geq 0} \Si_d(\fs) = \sum \limits_{d \geq 0} \ \sum\limits_{d=d_1> \dots > d_n\geq 0} \dfrac{1}{\ell_{d_1}^{s_1} \dots \ell_{d_n}^{s_n}}   \in K_{\infty}.
\end{equation*}
We set
$\Li(\emptyset)  = 1$. We call $\depth(\fs) = n$ the depth and $w(\fs) = s_1 + \dots + s_n$ the weight of the presentation $\Li(\fs)$. We see easily that for all $a,b \in \mathbb N$,
\begin{equation} \label{eq:product depth1}
\Si_d(a) \Si_d(b)=\Si_d(a+b)
\end{equation}
Using this product formula for the product of two power sums of depth $1$, one can define the product of two power sums of arbitrary depth. We show:

\begin{proposition}  \label{prop: product CMPL}
Let $\mathfrak a=(a_1,\dots,a_r)$ and $\mathfrak b=(b_1,\dots,b_k)$ be two tuples of positive integers.
\begin{enumerate}[\normalfont (1)]
\item  There exist constants $f_i \in \F_q$ and tuples of positive integers $\mathfrak c_i$ with $\depth(\mathfrak c_i) \leq \depth(\mathfrak a)+\depth(\mathfrak b)$ for all $i$, such that for all $d \in \mathbb Z$,
\begin{equation*} \label{eq:Sid}
\Si_d(\mathfrak a) \, \Si_d(\mathfrak b)=\sum_i f_i \Si_d(\mathfrak c_i).
\end{equation*}
\item There exist constants $f_i' \in \F_q$ and tuples of positive integers $\mathfrak c_i'$ with $\depth(\mathfrak c_i') \leq \depth(\mathfrak a)+\depth(\mathfrak b)$ for all $i$, such that for all $d \in \mathbb Z$,
\begin{equation} \label{eq:Sid minus}
\Si_{<d}(\mathfrak a) \Si_{<d}(\mathfrak b)=\sum_i f_i' \Si_{<d}(\mathfrak c_i').
\end{equation}
\item There exist $f''_i \in \mathbb{F}_q$ and tuples $\mathfrak{c}''_i$ with  $\depth(\mathfrak{c}''_i) \leq \depth(\mathfrak a) + \depth(\mathfrak b)$ for all $i$  such that for all $d \in \mathbb Z$,
    \begin{equation*}
        \Si_d(\mathfrak a) \Si_{<d}(\mathfrak b) = \sum \limits_i f''_i \Si_d(\mathfrak{c}''_i)  \quad \text{for all } d \in \mathbb{Z}.
    \end{equation*}
\end{enumerate}
\end{proposition}

We keep the notation as in Eq. \eqref{eq:Sid minus} and write
\begin{equation} \label{eq: star}
\mathfrak a*\mathfrak b=\sum_i f_i' \mathfrak c_i'.
\end{equation}

By Proposition~\ref{prop: product CMPL}, the $\Fq$-vector space (resp. $K$-vector space) generated by the CMPL's is an $\Fq$-algebra (resp. $K$-algebra).

We end this section by pointing out useful identities among these power sums.
\begin{lemma} \label{lemma:identities S Si}
For $\fs=(s_1,\ldots,s_r) \in \mathbb N^r$ with $1 \leq s_1,\ldots,s_r \leq q$, we have $S_d(\fs)=\Si_d(\fs)$. In particular, $\zeta_A(\fs)=\Li(\fs)$.
\end{lemma}

\begin{proof}
See \cite[\S 3.1-3.3]{Tha09}.
\end{proof}

We now state Brown's theorem for CMPL's. By Lemma \ref{lemma:identities S Si} the fundamental relation $R_1$ given in~\eqref{eq:R1} can be expressed as follows:
\begin{equation} \label{eqn: fundamental relation Si}
\Si_d(q)+D_1 \Si_{d+1}(1,q-1)=0,
\end{equation}
where we recall that $D_1=\theta^q-\theta \in K$.

For $w \in \N$ we recall (see Definition \ref{defn: basis MZV}) that $\mathcal J_w$ is the set of tuples $\fs=(s_1,\ldots,s_r)$ of weight $w$ with $1\leq s_i\leq q$ for $1\leq i\leq r-1$ and $s_r<q$. In \cite{IKLNDP24} we follow the same line as in the proof of \cite[Theorem A]{ND21} and show:

\begin{theorem} \label{thm: algebraic CMPL}
Let $w \in \N$. Then every CMPL of weight $w$ can be written as an $\Fq[D_1]$-linear combination of CMPL's $\Li(\fs)$ with $\fs \in \mathcal J_w$.
\end{theorem}

We stress that there is an explicit algorithm which expresses an arbitrary CMPL as a linear combination over $K$ of CMPL's $\Li(\fs)$ with $\fs \in \mathcal J_w$. This algorithm will be revisited in \S \ref{sec: algorithm revisited}.

\subsection{Fundamental connection between MZV's and CMPL's} \ppar

As already mentioned in the introduction 
we recall the following connection (see \cite[Theorem~5.9]{IKLNDP24}) which plays a central role in the proof of Zagier-Hoffman's conjectures in positive characteristic in \cite[Theorem A]{IKLNDP24} (see also~\cite{CCM23}):

\begin{theorem} \label{thm: connection}
Let $w \in \N$. Then the $K$-vector spaces spanned by MZV's and CMPL's of weight $w$ coincide.
\end{theorem}

\begin{proof}
By Theorems \ref{thm: generating set}, \ref{thm: algebraic CMPL} and Lemma \ref{lemma:identities S Si}, these two vector spaces are spanned by the generating set $\mathcal T_w$ given as in Theorem \ref{thm: generating set}. Therefore, they coincide.
\end{proof}

Then the Zagier-Hoffman's conjectures in positive characteristic were proved in \cite[Theorem A]{IKLNDP24} (see also \cite{CCM23}):
\begin{theorem} \label{thm: Zagier Hoffman}
Let $w \in \N$. Then the $K$-vector spaces spanned by MZV's and CMPL's of weight $w$ have the basis consisting of CMPL's $\Li(\fs)$ with $\fs \in \mathcal J_w$.
\end{theorem}


\section{Proof of Theorem \ref{thmx: new connection}} \label{sec: bridge}

In this section we prove Theorem \ref{thmx: new connection} which is a generalization of Theorem \ref{thm: connection} and states that the $\Fq$-spans of MZV's and CMPL's coincide (see Theorem \ref{thm: new connection} below).  The method used in the proof of Theorem \ref{thm: connection} in \cite{IKLNDP24} is based on that of \cite{ND21} and does not extend to $\Fq$. Instead, our proof of Theorem \ref{thmx: new connection}  is more direct and uses completely different techniques from those developed in \cite{ND21}, in particular, avoiding analogues of Brown's theorems, i.e., Theorems~\ref{thm: generating set} and \ref{thm: algebraic CMPL}.

\subsection{Statement of the result} \ppar

The main result of this section (Theorem \ref{thmx: new connection}) reads as follows:

\begin{theorem} \label{thm: new connection}
Let $w \in \N$. Then,
\begin{enumerate}[\normalfont (1)]
\item Every MZV of weight $w$ can be written as a linear combination over $\Fq$ of CMPL's.
\item Every CMPL of weight $w$ can be written as a linear combination over $\Fq$ of MZV's.
\end{enumerate}

In particular, the vector spaces over $\Fq$ spanned by MZV's and CMPL's of weight~$w$ coincide.
\end{theorem}

The rest of this section is devoted to proving this theorem. We just mention that Theorem \ref{thm: new connection} immediately implies Theorem \ref{thm: connection} and could be seen as a stronger version of Theorem \ref{thm: connection}.

\subsection{Proof of Theorem \ref{thm: new connection}} \label{sec: proof main result} \ppar

\subsubsection{}
First, we show:
\begin{proposition} \label{prop: MZV}
Let $\fs$ be a tuple of weight $w$. Then there exist $a_i \in \mathbb{F}_q$ and tuples~$\fs_i$ of weight $w$ such that for all $d \in \mathbb N$,
\[
S_d(\fs)=\sum _i a_i \Si_d(\fs_i).
\]
\end{proposition}

\begin{proof}
The proof is by induction on $w$. If $w=0$, then $\fs=\emptyset$ and we are done. If $w=1$, then it follows from Lemma \ref{lemma:identities S Si}. We assume $w \geq 2$ and suppose that for all tuples $\fs'$ of weight $w(\fs')=w'<w$, there exist $a_i' \in \mathbb{F}_q$ and tuples $\fs_i'$ of weight $w'$ such that for all $d \in \mathbb N$,
\[
S_d(\fs')=\sum _i a_i' \Si_d(\fs_i').
\]

We claim that for all tuples $\fs$ of weight $w$, there exist $a_i \in \mathbb{F}_q$ and tuples $\fs_i$ of weight $w$ such that for all $d \in \mathbb N$,
\[
S_d(\fs)=\sum _i a_i \Si_d(\fs_i).
\]
In fact, we first consider the case where $\depth(\fs) \geq 2$. We write $\fs=(s_1,\fs_-)$. By the induction hypothesis, there exist $b_i,c_j \in \mathbb{F}_q$ and tuples $\ft_i, \mathfrak{u}_j$ of weight $s_1$ and $w(\fs_-)$ such that for all $d \in \mathbb N$,
\begin{align*}
S_d(s_1) =\sum _i b_i \Si_d(\ft_i), \quad \text{ and } \quad
S_d(\fs_-) =\sum _j c_j \Si_d(\mathfrak{u}_j).
\end{align*}
Thus
\[
S_{<d}(\fs_-) =\sum _j c_j \Si_{<d}(\mathfrak{u}_j).
\]
It follows that
\begin{align*}
S_d(\fs)=S_d(s_1) S_{<d}(\fs_-)=\sum _i b_i \Si_d(\ft_i) \sum _j c_j \Si_{<d}(\mathfrak{u}_j).
\end{align*}
So  we are done by Proposition \ref{prop: product CMPL}.

To conclude, we have to prove the claim for $\fs=(w)$. By Eqs. \eqref{eq: sum MZV} and \eqref{eq:Delta Chen},
\begin{equation*}
S_d(w)= S_d(w-1) S_d(1)  - \sum \limits_j \Delta^i_{w-1,1} S_d(w-j, j)
\end{equation*}
with $\Delta^i_{w-1,1} \in \Fq$. By Lemma \ref{lemma:identities S Si}, we know that $S_d(1)=\Si_d(1)$. By induction hypothesis and the previous discussion, there exist $b_i' ,c_k' \in \mathbb{F}_q$ and tuples $\ft_i', \mathfrak{u}_k'$ of weight~$w-1$ and~$w$ such that for all $d \in \mathbb N$,
\begin{align*}
S_d(w-1) &=\sum _i b_i' \Si_d(\ft_i'), \\
\sum \limits_j \Delta^i_{w-1,1} S_d(w-j, j) &=\sum _k c'_k \Si_d(\mathfrak{u}_k).
\end{align*}
It follows that
\begin{align*}
S_d(w) = \sum _i b_i' \Si_d(\ft_i') \Si_d(1) - \sum _k c'_k\Si_d(\mathfrak{u}_k).
\end{align*}
By Proposition \ref{prop: product CMPL} again, we are done.
\end{proof}

\subsubsection{}
Next, we prove:
\begin{proposition} \label{prop: CMPL}
Let $\fs$ be a tuple of weight $w$. Then there exist $a_i \in \mathbb{F}_q$ and tuples~$\fs_i$ of weight $w$ such that for all $d \in \mathbb N$,
\[
\Si_d(\fs)=\sum _i a_i S_d(\fs_i).
\]
\end{proposition}

\begin{proof}
The proof is by induction on $w$. If $w=0$, then $\fs=\emptyset$ and we are done. If $w=1$, then it follows from Lemma \ref{lemma:identities S Si}. We assume $w \geq 2$ and suppose that for all tuples $\fs'$ of weight $w(\fs')=w'<w$, there exist $a_i' \in \mathbb{F}_q$ and tuples $\fs_i'$ of weight $w'$ such that for all $d \in \mathbb N$,
\[
\Si_d(\fs')=\sum _i a_i' S_d(\fs_i').
\]

We claim that for all tuples $\fs$ of weight $w$, there exist $a_i \in \mathbb{F}_q$ and tuples $\fs_i$ of weight $w$ such that for all $d \in \mathbb N$,
\[
\Si_d(\fs)=\sum _i a_i S_d(\fs_i).
\]
In fact, we first consider the case where $\depth(\fs) \geq 2$. We write $\fs=(s_1,\fs_-)$. By the induction hypothesis, there exist $b_i,c_j \in \mathbb{F}_q$ and tuples $\ft_i, \mathfrak{u}_j$ of weight $s_1$ and $w(\fs_-)$ such that for all $d \in \mathbb N$,
\begin{align*}
\Si_d(s_1) =\sum _i b_i S_d(\ft_i), \quad \text{ and } \quad
\Si_d(\fs_-) =\sum _j c_j S_d(\mathfrak{u}_j).
\end{align*}
Thus
\[
\Si_{<d}(\fs_-) =\sum _j c_j S_{<d}(\mathfrak{u}_j).
\]
It follows that
\begin{align*}
\Si_d(\fs)=\Si_d(s_1) \Si_{<d}(\fs_-)=\sum _i b_i S_d(\ft_i) \sum _j c_j S_{<d}(\mathfrak{u}_j).
\end{align*}
So we are done by Proposition \ref{prop: product MZV}.

To conclude, we have to prove the claim for $\fs=(w)$. By Eq.  \eqref{eq:product depth1},
\begin{equation*}
\Si_d(w)= \Si_d(w-1) \Si_d(1).
\end{equation*}
By Lemma \ref{lemma:identities S Si}, we know that $\Si_d(1)=S_d(1)$. By induction hypothesis, there exist $b_i' \in \mathbb{F}_q$ and tuples $\ft_i'$ of weight $w-1$ such that for all $d \in \mathbb N$,
\begin{align*}
\Si_d(w-1) &=\sum _i b_i' S_d(\ft_i').
\end{align*}
It follows that
\begin{align*}
\Si_d(w) = \sum _i b_i' S_d(\ft_i') S_d(1).
\end{align*}
By Proposition \ref{prop: product MZV} again, we are done.
\end{proof}

\subsubsection{}
Finally, Theorem \ref{thm: new connection} follows immediately from Propositions \ref{prop: MZV} and \ref{prop: CMPL}.


\section{Algorithm revisited and a weight-preserving operator} \label{sec: algorithm revisited}

In~\cite{ND21} the third author established an analogue of Brown's theorem in \cite{Bro12}. 
The key ingredient is an algorithm expressing any MZV as a $K$-linear combination of 
MZV's that are larger with respect to a suitable ordering. Notably, in the classical 
setting, Brown's theorem guarantees the existence of such an expression but does not 
provide a constructive algorithm (see \cite{Bro12,Del13}). We mention that a characteristic zero analogue of 
this algorithm was recently investigated in~\cite{HMSW25}. We refer the reader to \cite{Bro12,Del13,HMSW25} for a more comprehensive discussion.

In this section we revisit the algorithmic framework introduced in~\cite{ND21} 
and~\cite{IKLNDP24} (see also~\cite{CCM23}). As an application, we present an 
alternative proof of Theorem~\ref{thm: algebraic CMPL} which serves as an analogue 
of Brown's theorem for CMPL's. More importantly, this algorithm gives rise to a new 
weight-preserving operator $\Ind$ on $\Fq$-linear relations among CMPL's 
(see~\S\ref{sec: new operator}).

\subsection{Binary relations} \ppar \label{sec: binary relations}

We adopt the notation of binary relations from \cite{Tod18} (see also \cite{IKLNDP24, ND21}).
\begin{definition} \label{defn: binary relations}\ppar

\begin{enumerate}[\normalfont (1)]
\item {\it A binary relation} is an $\mathbb{F}_q[D_1]$-linear combination of the form
\[ \sum_{i}a_i \Si_d(\fs_i) + \sum_j b_j \Si_{d+1}(\ft_j)=0\quad \text{for  $d\in \mathbb{Z}$},\]
where $a_i, b_j \in \mathbb{F}_q[D_1]$ and $\fs_i$, $\ft_j$ are tuples of the same weight.
\item  A binary relation is said to be {\it a fixed relation} if $b_j =0$ for all $j$.
\end{enumerate}
\end{definition}

Let $\mathfrak{BR}_w$ be the set of all binary relations of weight $w$. Also, we use the notation $R(d)$ for the specialization at the finite level,
$$\sum a_i\Si_d(\fs_i) + \sum b_j \Si_{d+1}(\ft_j)=0$$
for a fixed $d\in \mathbb{Z}$.

The relation given in Eq. \eqref{eqn: fundamental relation Si}
\begin{align*}
  R_1: \quad \Si_d(q) + D_1\Si_{d+1}(1,q-1)=0 \quad\text{for all $d\in\mathbb{Z}$,}
\end{align*}
is an example of binary relations of weight $q$.

We also recall the operators $\mathcal{B}^*$ and $\mathcal{C}$ acting on binary relations (see \cite{ND21,Tod18}). 
Let $n>0$ be an integer and let $R \in \mathfrak{BR}_w$. We define $\mathcal{B}_n^*\colon \mathfrak{BR}_{w}\to \mathfrak{BR}_{w+n}$ given by
$$\mathcal{B}^*_n(R)(d) = \Si_d(n)\sum_{j<d}R(j).$$
This extends to
$$\mathcal{B}^*_{(n_1, \dots, n_\ell)} := \mathcal{B}^*_{n_1}\circ \dots \circ \mathcal{B}^*_{n_\ell},$$
for positive integers $n_i$.

For a tuple of positive integers $\mathfrak{m} = (m_1, \dots, m_\ell)$, we also define $\mathcal{C}_{\mathfrak{m}}\colon \mathfrak{BR}_{w}\to \mathfrak{BR}_{w+ w(\mathfrak{m})}$ given by
$$\mathcal{C}_{\mathfrak{m}}(R)(d) = R(d)\Si_{<d+1}(\mathfrak{m}).$$

\subsection{A modified key expression} \ppar \label{sec: key expression}

Let $\fs=(s_1,\dots,s_\ell)$ be a tuple of weight $w$. We suppose that $\fs$ does not belong to the set $\mathcal J_w$ defined as in Definition \ref{defn: basis MZV}. Then there exists
\begin{itemize}
\item either $1 \leq k \leq \ell$ such that $1 \leq s_1,\dots,s_{k-1} \leq q$ and $s_k>q$,

\item or $1 \leq s_1,\dots,s_{\ell-1} \leq q$ and $s_\ell=q$.
\end{itemize}

We express
\begin{equation} \label{eq: decomposition}
\fs=(\frak n,q+r,\frak m),
\end{equation}
where
\begin{itemize}
\item $\frak n=(s_1,\dots,s_{k-1})$ with $1 \leq s_1,\dots,s_{k-1} \leq q$,
\item either $(r,\frak m)$ is a tuple of positive integers or $(r,\frak m)=\emptyset$ (i.e., $r=0$ and $\frak m=\emptyset$).
\end{itemize}

When $\mathfrak{n} \ne \emptyset$, we set
$$\frak n^+:=(s_1,\dots,s_{k-1}+1).$$

\begin{definition} \label{def: types}
Let $\fs$ be a tuple as given above, that means $\fs=(s_1,\dots,s_\ell)$ be a tuple of weight $w$ which  does not belong to the set $\mathcal J_w$ defined as in Definition \ref{defn: basis MZV}.

We express $\fs=(\frak n,q+r,\frak m)$ as in Eq. \eqref{eq: decomposition} and define the \textsf{Type} of $\fs$ as follows:
\begin{align*}
 \textsf{(Type 0)}&& \fs&= (w), \\
 \textsf{(Type 1)}&&\fs&= (\mathfrak{n}, q), &&  \text{ with }\depth(\mathfrak{n})\ge1, \\
 \textsf{(Type 2)}&&\fs&= (\mathfrak{n}, q+r), &&   \text{ with } r\ge1, \ \depth(\mathfrak{n})\ge1,\\
 \textsf{(Type 3)}&&\fs&= (q+r, \mathfrak{m}), &&   \text{ with } r\ge1, \ \depth(\mathfrak{m})\ge1,\\
 \textsf{(Type 4)}&&\fs&= (\mathfrak{n}, q+r, \mathfrak{m}), &&  \text{ with } r\ge1, \ \depth(\mathfrak{n})\ge1, \ \depth(\mathfrak{m})\ge1.
\end{align*}
\end{definition}

We slightly modify the algorithm given as in \cite{ND21}. We distinguish three cases depending on the \textsf{Type} of $\fs$.

\subsubsection{Case 1: $\fs$ is of \textsf{Type 2} or \textsf{Type 4}. } \ppar

In this case, $\fn \ne \emptyset$ and $r \ge 1$, and $\fs=(s_1,\dots,s_\ell)=(s_1,\dots,s_{k-1},q+r,\frak m)$ with $r\geq 1$, $k\geq 2$ and possibly having $\frak m = \emptyset$. First, $\mathcal{C}_{(r,\frak m)}(R_1)$ yields
\begin{align*}
0 &= (\Si_d(q) + D_1\Si_{d+1}(1,q-1)) \Si_{<d+1}(r,\frak m) \\
&= \Si_d(q+r,\frak m)+\Si_d(q,r,\frak m)+D_1 \Si_{d+1}(1,(q-1)*(r,\frak m)).
\end{align*}
Here the operation $*$ is given as in Eq. \eqref{eq: star}.

Next, we apply $\mathcal{B}^*_{s_{k-1}}$ to get a fixed relation
\begin{align*}
0 = & \Si_d(s_{k-1})(\Si_{<d}(q+r,\frak m)+\Si_{<d}(q,r,\frak m)+D_1 \Si_{<d+1}(1,(q-1)*(r,\frak m))) \\
= &\Si_d(s_{k-1},q+r,\frak m)+\Si_d(s_{k-1},q,r,\frak m) \\
& +D_1 \Si_d(s_{k-1}+1,(q-1)*(r,\frak m)) +D_1 \Si_d(s_{k-1},1,(q-1)*(r,\frak m)).
\end{align*}

Finally, we apply $\mathcal{B}^*_{s_1}\circ \dots \circ \mathcal{B}^*_{s_{k-2}}$ to get
\begin{align*}
0 = & \Si_d(s_1,\dots,s_{k-1},q+r,\frak m)+\Si_d(s_1,\dots,s_{k-1},q,r,\frak m) \\
& +D_1 \Si_d(s_1,\dots,s_{k-1}+1,(q-1)*(r,\frak m)) +D_1 \Si_d(s_1,\dots,s_{k-1},1,(q-1)*(r,\frak m)).
\end{align*}

Thus taking the sum over all $d$ yields
\begin{align*}
0 = &\Li(\fs)+\Li(s_1,\dots,s_{k-1},q,r,\frak m) \\
& +D_1 \Li(s_1,\dots,s_{k-1}+1,(q-1)*(r,\frak m)) +D_1 \Li(s_1,\dots,s_{k-1},1,(q-1)*(r,\frak m)).
\end{align*}

\subsubsection{Case 2: $\fs$ is of \textsf{Type 0} or \textsf{Type 3}.} \ppar

In this case, $\fn = \emptyset$, and $\fs=(q+r,\frak m)$ with $r\geq 1$ and possibly having $\frak m = \emptyset$. We note that $\mathcal{C}_{(r,\frak m)}(R_1)$ yields
\begin{align*}
0 &= (\Si_d(q) + D_1\Si_{d+1}(1,q-1)) \Si_{<d+1}(r,\frak m)  \\
&= \Si_d(q+r,\frak m)+\Si_d(q,r,\frak m)+D_1 \Si_{d+1}(1,(q-1)*(r,\frak m)).
\end{align*}

Thus taking the sum over all $d$ yields
\begin{align*}
0 =\Li(\fs)+\Li(q,r,\frak m) +D_1 \Li(1,(q-1)*(r,\frak m)).
\end{align*}

\subsubsection{Case 3: $\fs$ is of \textsf{Type 1}.} \ppar

In this case, $\fn \ne \emptyset$, $r = 0$ and $\fm=\emptyset$, and $\fs=(s_1,\dots,s_{k-1},q)$ with $k\geq 2$.

First, we directly apply $\mathcal{B}^*_{s_{k-1}}$ to the relation $R_1$ and get a fixed relation
\begin{align*}
0 &= \Si_d(s_{k-1})(\Si_{<d}(q)+D_1 \Si_{<d+1}(1,q-1)) \\
&= \Si_d(s_{k-1},q) +D_1 \Si_d(s_{k-1}+1,q-1) +D_1 \Si_d(s_{k-1},1,q-1).
\end{align*}

Next, we apply $\mathcal{B}^*_{s_1}\circ \dots \circ \mathcal{B}^*_{s_{k-2}}$ to get
\begin{align*}
0 = & \Si_d(s_1,\dots,s_{k-1},q) \\
& +D_1 \Si_d(s_1,\dots,s_{k-1}+1,q-1) +D_1 \Si_d(s_1,\dots,s_{k-1},1,q-1).
\end{align*}

Thus taking the sum over all $d$ yields
\begin{align*}
0 = & \Li(\fs) \\
&  +D_1 \Li(s_1,\dots,s_{k-1}+1,q-1) +D_1 \Li(s_1,\dots,s_{k-1},1,q-1).
\end{align*}

The following proposition summarizes the previous discussion.
\begin{proposition} \label{prop: key expression}
Let $\fs=(s_1,\dots,s_\ell)$ be a tuple of weight $w$ which does not belong to $\mathcal J_w$  as in Definition \ref{defn: basis MZV} and is given by
\begin{equation*}
\fs=(\frak n,q+r,\frak m),
\end{equation*}
where
\begin{itemize}
\item $\frak n=(s_1,\dots,s_{k-1})$ with $1 \leq s_1,\dots,s_{k-1} \leq q$,
\item either $(r,\frak m)$ is a tuple of positive integers or $(r,\frak m)=\emptyset$ (i.e., $r=0$ and $\frak m=\emptyset$).
\end{itemize}
We recall
$$\frak n^+=(s_1,\dots,s_{k-1}+1).$$
Then
\begin{align}
\begin{split}
  \Li(\fs) =& -\mathbf{1}_{r\ne0}\cdot \Li(\fn, q, r, \fm) \\
  & - D_1 \mathbf{1}_{\fn \ne \emptyset}\cdot \Li(\fn^{+}, (q-1)*(r, \fm)) - D_1 \Li(\fn, 1, (q-1)*(r,\fm)).\end{split} \label{eq: revisted expression}
\end{align}
\end{proposition}

We emphasize an important difference with the algorithm presented in \cite{ND21}. The attentive reader may notice that we do not use the operator $\mathcal{BC}$. Further, in the key expression for $\Li(\fs)$ we have only two levels:
\begin{itemize}
\item the $\Fq$-level tuples: $\mathbf{1}_{r\ne0}\cdot (\fn, q, r, \fm)$
\item the $D_1$-level tuples: $\mathbf{1}_{\fn \ne \emptyset}\cdot (\fn^{+}, (q-1)*(r, \fm))$ and $(\fn, 1, (q-1)*(r,\fm))$.
\end{itemize}

\subsection{The $(\depth,\lex)$-order and a modified algorithm } \ppar

We introduce a new order which is suitable for the previous algorithm.
\begin{definition} \label{defn: order}
Let $\fs$ and $\ft$ be two tuples of the same weight. 
\begin{enumerate}[\normalfont (1)]
\item We define a new order $<$ called {\it the $(\depth,\lex)$-order} on tuples as follows: We say that $\fs < \ft$ in the $(\depth,\lex)$-order if
\begin{itemize}
\item either $\depth(\ft) > \depth(\fs)$,
\item or $\depth(\ft)= \depth(\fs)$ and $\ft \succ \fs$ in the lexicographical order.
\end{itemize}

\item We say that $\fs \leq \ft$ if either $\fs=\ft$ or $\fs<\ft$.

\item Finally, we say that $\ft > \fs$ (resp. $\ft \geq \fs$) if $\fs < \ft$ (resp. $\fs \leq \ft$).
\end{enumerate}
\end{definition}

We obtain immediately the following lemma which implies that the algorithm must terminate after a finite number of steps.
\begin{lemma} \label{lemma: order}
We keep the notation as in Proposition \ref{prop: key expression} and consider Eq. \eqref{eq: revisted expression}. Recall that $\fs=(\fn,q+r,\fm)$. Let $\ft$ be a tuple among the terms which occur in the expansion of Eq. \eqref{eq: revisted expression}:
\begin{itemize}
\item $\mathbf{1}_{r\ne0}\cdot (\fn,q,r, \fm)$,
\item $\mathbf{1}_{\fn \ne \emptyset}\cdot (\fn^{+},(q-1)*(r, \fm))$,
\item $(\fn,1,(q-1)*(r, \fm))$.
\end{itemize}
Then we have $\ft>\fs$.
\end{lemma}

\begin{proof}
We recall that for any nonempty tuple $\mathfrak a$, $(q-1)*\mathfrak a$ is written as in Eq. \eqref{eq: star}:
\begin{equation*} 
(q-1)*\mathfrak a=\sum_i f_i \mathfrak a_i.
\end{equation*} 
Here $f_i \in \Fq$ and tuples $\mathfrak a_i$ are given by Proposition \ref{prop: product CMPL}. Thus for all $i$, either $\depth{\mathfrak a_i}=\depth(\mathfrak a)$ or $\depth{\mathfrak a_i}=\depth(\mathfrak a)+1$. The lemma follows immediately.
\end{proof}

\begin{proof}[Proof of Theorem \ref{thm: algebraic CMPL} using Proposition \ref{prop: key expression}] \ppar

We now use Proposition \ref{prop: key expression} and present an alternative proof of Theorem \ref{thm: algebraic CMPL}. The proof is by descending induction on $\fs$ of weight $w$ with respect to the $(\depth,\lex)$-order $<$ on tuples.

First, we consider $\fs=(\{1\}^{w})$ which is the largest tuple of weight $w$ with respect to the order $<$, then it belongs to $\mathcal J_w$ and we are done.

Next, we fix a tuple $\fs$ of weight $w$. We assume that for any tuple $\ft$ of weight $w$ such that $\ft>\fs$, we can express $\Li(\ft)$ as an $\Fq[D_1]$-linear combination of CMPL's $\Li(\ft')$ with $\ft' \in \mathcal J_w$. In fact, if $\fs$ belongs to $\mathcal J_w$, then we are done. Otherwise, by Proposition \ref{prop: key expression}, we express $\Li(\fs)$ as an $\Fq[D_1]$-linear combination of CMPL's $\Li(\ft)$ with $\ft > \fs$. Then we are done by induction hypothesis.
\end{proof}

We end this section by another consequence of the previous modified key expression which plays the central role in the sequel.

\begin{proposition} \label{prop: key expression 2}
Let $\fs=(s_1,\dots,s_\ell)$ be a tuple of weight $w$ which does not belong to $\mathcal J_w$  as in Definition \ref{defn: basis MZV} such that
\begin{equation*}
\fs=(\frak n,q+r,\frak m),
\end{equation*}
where
\begin{itemize}
\item $\frak n=(s_1,\dots,s_{k-1})$ with $1 \leq s_1,\dots,s_{k-1} \leq q$,
\item either $(r,\frak m)$ is a tuple of positive integers or $(r,\frak m)=\emptyset$ (i.e., $r=0$ and $\frak m=\emptyset$).
\end{itemize}
When $\mathfrak{n} \ne \emptyset$, we recall
$$\frak n^+=(s_1,\dots,s_{k-1}+1).$$

More generally, for any tuple $\ft$ of weight $w$ not in $\mathcal J_w$, we denote by $\fn_{\ft}, r_{\ft},$ and~$ \fm_{\ft}$ the data associated with $\ft$, so that $\ft=(\frak n_{\ft},q+r_{\ft},\frak m_{\ft})$.

Then there exist constants $b_{\fu}, a_{\ft} \in \Fq$ such that
\begin{align}
\begin{split}
  \Li(\fs)=& \sum_{\fu>\fs, \, \fu \in \mathcal J_w} b_{\fu} \Li(\fu) \\
  &- D_1 \mathbf{1}_{\fn \ne \emptyset}\cdot \Li(\fn^{+}, (q-1)*(r, \fm)) - D_1 \Li(\fn, 1, (q-1)*(r,\fm)) \\
  &+ \sum_{\ft>\fs, \, \ft \notin \mathcal J_w} D_1 a_{\ft} \left (\mathbf{1}_{\fn_{\ft} \ne \emptyset}\cdot \Li(\fn_{\ft}^{+}, (q-1)*(r_{\ft}, \fm_{\ft})) + \Li(\fn_{\ft}, 1, (q-1)*(r_{\ft},\fm_{\ft})) \right).\end{split} \label{eq: D_1 expression}
\end{align}
\end{proposition}

\begin{proof}
We apply Proposition \ref{prop: key expression} for $\Li(\fs)$. By the previous discussion, there are two levels for the tuples appearing in the expression: the $\Fq$-level tuples and the $D_1$-level tuples. We want to eliminate the $\Fq$-level tuples by applying Proposition~\ref{prop: key expression} again. Suppose that  there exists an $\Fq$-level tuple $\Li(\ft)$ such that $\ft \notin \mathcal J_w$. Then we apply Proposition \ref{prop: key expression} for $\ft$, and so on. This will terminate after a finite number of steps by Lemma \ref{lemma: order}. At the end, we get
\begin{itemize}
\item either the $\Fq$-level tuples which give rise to the terms $\sum_{\fu>\fs, \, \fu \in \mathcal J_w} a_{\fu} \Li(\fu)$;

\item or the $D_1$-level tuples which give rise to the remaining terms.
\end{itemize}
Therefore, we are done.
\end{proof}

\subsection{An operator on $\Fq$-linear relations} \ppar \label{sec: new operator}

In this section we introduce an operator $\Ind$ on $\Fq$-linear relations among CMPL's of weight $w$. The key ingredient is the modified key expression as presented in Proposition~\ref{prop: key expression 2}. We note that this operator preserves the weight $w$ in contrast with the operators $\mathcal B^*$ and $\mathcal C$ on binary relations in \S \ref{sec: binary relations}.

\begin{definition}\ppar
\begin{enumerate}[\normalfont (1)]
\item We define an {\it $\Fq$-linear relation} to be a relation given by
\begin{equation} \label{eq: Fq relation}
\mathcal{R}: \quad \sum_i a_i \Li(\fs_i) =0,
\end{equation}
where $a_i \in \Fq$ and $\fs_i$ are tuples of weight $w$.
\item  We say that this $\Fq$-linear relation $\mathcal{R}$ is {\it nontrivial} if not all $a_i$ are zero.
\item  Let $\mathcal{R}$ be a nontrivial $\Fq$-linear relation. Then we define $s(\mathcal{R})$ to be the smallest tuple in the (depth-lex)-order among the tuples $\mathfrak{s}_i$’s with the condition $a_i \ne 0$ in Eq. \eqref{eq: Fq relation}.
    \end{enumerate}
\end{definition}

Let $\mathcal{R}$ be a nontrivial $\Fq$-linear relation as in Eq. \eqref{eq: Fq relation}. We suppose that $a_i \ne 0$ for all $i$ and put $\fs=s(\mathcal R)$ which is the smallest tuple among $\fs_i$. Finally we express $\mathcal R$ as follows:
$$\Li(\fs)+\sum_{\fs<\fs_i} a_i \Li(\fs_i) =0.$$

By Proposition \ref{prop: key expression 2}, there exist constants $b_{\fu}, a_{\ft} \in \Fq$ such that
\begin{align} \label{eq: induced relation}
\begin{split}
  0=& \sum_{\fu>\fs, \, \fu \in \mathcal J_w} b_{\fu} \Li(\fu) \\
  &+ D_1 \mathbf{1}_{\fn \ne \emptyset}\cdot \Li(\fn^{+}, (q-1)*(r, \fm)) + D_1 \Li(\fn, 1, (q-1)*(r,\fm)) \\
  &+ \sum_{\ft>\fs, \, \ft \notin \mathcal J_w} D_1 a_{\ft} \left (\mathbf{1}_{\fn_{\ft} \ne \emptyset}\cdot \Li(\fn_{\ft}^{+}, (q-1)*(r_{\ft}, \fm_{\ft})) + \Li(\fn_{\ft}, 1, (q-1)*(r_{\ft},\fm_{\ft})) \right).\end{split}
\end{align}

We claim that
\begin{equation} \label{eq: Fq terms}
\sum_{\fu>\fs, \, \fu \in \mathcal J_w} b_{\fu} \Li(\fu)=0.
\end{equation}
In fact, suppose that it is not the case. By Theorem \ref{thm: algebraic CMPL}, we can express the $D_1$-level terms as $\Fq[D_1]$-linear combinations of CMPL's $\Li(\ft)$ with $\ft \in \mathcal J_w$. It follows that there exist constants $c_{\ft} \in \Fq[D_1]$ such that
\begin{align*}
\begin{split}
  0= \sum_{\fu>\fs, \, \fu \in \mathcal J_w} b_{\fu} \Li(\fu) + \sum_{\ft>\fs, \, \ft \in \mathcal J_w} D_1 c_{\ft} \Li(\ft).\end{split}
\end{align*}
By Theorem \ref{thm: Zagier Hoffman} we know that $\Li(\ft)$ with $\ft \in \mathcal J_w$ are linearly independent over $K$. But we obtain a contradiction by looking at the coefficient of $\Li(\fu)$ with $b_\fu \ne 0$.

Combining Eqs. \eqref{eq: induced relation} and \eqref{eq: Fq terms} and dividing by $D_1$ yield the following $\Fq$-linear relation which we denote by $\Ind(\mathcal R)$:
\begin{align} \label{eq: induced relation 2}
\begin{split}
  0=& \mathbf{1}_{\fn \ne \emptyset}\cdot \Li(\fn^{+}, (q-1)*(r, \fm)) +  \Li(\fn, 1, (q-1)*(r,\fm)) \\
  &+ \sum_{\ft>\fs, \, \ft \notin \mathcal J_w} a_{\ft} \left (\mathbf{1}_{\fn_{\ft} \ne \emptyset}\cdot \Li(\fn_{\ft}^{+}, (q-1)*(r_{\ft}, \fm_{\ft})) + \Li(\fn_{\ft}, 1, (q-1)*(r_{\ft},\fm_{\ft})) \right).\end{split}
\end{align}

We end this section with some remarks. First, we warn the reader that this $\Fq$-linear relation $\Ind(\mathcal R)$ may be trivial. Next, suppose that $\Ind(\mathcal R)$ is nontrivial. We note that all tuples in this new $\Fq$-linear relation $\Ind(\mathcal R)$ are strictly greater than~$\fs$. It follows that
$$s(\Ind(\mathcal R))>\fs=s(\mathcal R).$$

Therefore, we obtain:
\begin{proposition} \label{prop: trivial induced relation}
We keep the above notation. Let $w \in \mathbb N$. We consider the set of nontrivial $\Fq$-linear relations among CMPL's of weight $w$. We assume this set is nonempty and denote by $\mathcal R$ the nontrivial $\Fq$-linear relation such that $s(\mathcal R)$ is the largest tuple among such $\mathbb{F}_q$-relations.

Then the induced $\Fq$-linear relation $\Ind(\mathcal R)$ given in~Eq.~\eqref{eq: induced relation 2}  is trivial.
\end{proposition}

\begin{proof}
Suppose that $\Ind(\mathcal R)$ is nontrivial. Then
$$s(\Ind(\mathcal R))>s(\mathcal R).$$
We obtain a contradiction by the choice of $\mathcal R$ and so the proof  is completed.
\end{proof}


\section{Linear relations of weights up to $2q+1$}  \label{sec: small weights}

In this section we keep the same notation as in the previous section, and prove the following theorem which will be used to prove Theorems \ref{thm: weight 2q+1} and \ref{thm: weight up to 2q}.
\begin{theorem} \label{thm: small weights}
Let $w \in \mathbb N$ such that $w \leq 2q+1$. Let $\mathcal R$ be a nontrivial $\Fq$-linear relation such that $\Ind(\mathcal R)$ is trivial. Then we have
$$s(\mathcal R)=(q+2,q-1).$$
\end{theorem}

As a direct consequence, we obtain:
\begin{corollary} \label{cor: small weights}
Let $w \in \mathbb N$ such that $w \leq 2q+1$. We consider the set of nontrivial $\Fq$-linear relations among CMPL's of weight $w$. We suppose that this set is non-empty and denote by $\mathcal R$ the nontrivial $\Fq$-linear relation such that $s(\mathcal R)$ is the largest tuple. Then we have
$$s(\mathcal R)=(q+2,q-1).$$
\end{corollary}

\begin{proof}
By Proposition \ref{prop: trivial induced relation} we know that $\Ind(\mathcal R)$ is trivial. Hence the corollary follows immediately from Theorem \ref{thm: small weights}.
\end{proof}

The rest of this section is devoted to proving Theorem \ref{thm: small weights}.

\subsection{Proof of Theorem \ref{thm: small weights}: setup} \ppar

We write
\begin{equation*}
\mathcal R: \quad \sum_i a_i \Li(\fs_i) =0,
\end{equation*}
where $a_i \in \Fq$ and $\fs_i$ are tuples of weight $w \leq 2q+1$ and put
$$\fs=s(\mathcal R).$$
We can suppose that $\mathcal R$ is given by
\[ \Li(\fs)+\sum_{\fs<\fs_i} a_i \Li(\fs_i) =0.\]
As before, we express
\begin{equation*}
\fs=(\frak n,q+r,\frak m),
\end{equation*}
where
\begin{itemize}
\item $\frak n=(s_1,\dots,s_{k-1})$ with $1 \leq s_1,\dots,s_{k-1} \leq q$,
\item either $(r,\frak m)$ is a tuple or $(r,\frak m)=\emptyset$ (i.e.,  $r=0$ and $\frak m=\emptyset$).
\end{itemize}
Since $w(\fs) \leq 2q+1$, it follows that
\begin{equation} \label{eq:bound for weight}
w(\fn)+w(r,\fm) \leq q+1.
\end{equation}
Recall that
$$\frak n^+=(s_1,\dots,s_{k-1}+1).$$

By Proposition \ref{prop: key expression},
\begin{align}
\begin{split}
  \Li(\fs) =& -\mathbf{1}_{r\ne0}\cdot \Li(\fn, q, r, \fm) \\
  & - D_1 \mathbf{1}_{\fn \ne \emptyset}\cdot \Li(\fn^{+}, (q-1)*(r, \fm)) - D_1 \Li(\fn, 1, (q-1)*(r,\fm)).\end{split} \label{eq: revisted expression 1}
\end{align}

By assumption, the following induced $\Fq$-linear relation $\Ind(\mathcal R)$ given as in Eq. \eqref{eq: induced relation 2} is trivial:
\begin{align} \label{eq: trivial induced relation}
\begin{split}
  0=& \mathbf{1}_{\fn \ne \emptyset}\cdot \Li(\fn^{+}, (q-1)*(r, \fm)) +  \Li(\fn, 1, (q-1)*(r,\fm)) \\
  &+ \sum_{\ft>\fs, \, \ft \notin \mathcal J_w} a_{\ft} \left (\mathbf{1}_{\fn_{\ft} \ne \emptyset}\cdot \Li(\fn_{\ft}^{+}, (q-1)*(r_{\ft}, \fm_{\ft})) + \Li(\fn_{\ft}, 1, (q-1)*(r_{\ft},\fm_{\ft})) \right)\end{split},
\end{align}
where $a_{\ft} \in \Fq$.

\subsection{Proof of Theorem \ref{thm: small weights}: strategy of the proof} \ppar

The strategy of the proof is as follows: We recall that the \textsf{Type} of $\fs$ is defined as in Definition \ref{def: types}. Following the \textsf{Type} of $\fs$ we will choose a tuple~$\widetilde{\ft}$  which appears in Eq. \eqref{eq: revisted expression 1} among  the $D_1$-level tuples, i.e., among
$$\mathbf{1}_{\fn \ne \emptyset}\cdot (\fn^{+}, (q-1)*(r, \fm)) +  (\fn, 1, (q-1)*(r,\fm)).$$
Then we show that in the corresponding equation \eqref{eq: trivial induced relation}, the coefficient of $\widetilde{\ft}$ is nonzero. This contradicts the fact that Eq. \eqref{eq: trivial induced relation} is trivial.

Before stating the key propositions to cover the cases of each of \textsf{Types} of $s(\mathcal R)$, we give an auxiliary and useful lemma.

\begin{lemma} \label{lemma: special case}
Assume that $\widetilde{\mathfrak t}$ is a tuple that satisfies that $\fs < \ft$ implies $\widetilde{\ft} \le \ft$. Then $\widetilde{\ft}$ cannot appear in Eq. \eqref{eq: revisted expression 1} among  the $D_1$-level tuples.
\end{lemma}

\begin{proof}
Suppose that $\widetilde{\ft}$ appears in Eq. \eqref{eq: revisted expression 1} among  the $D_1$-level tuples, i.e., among
$$\mathbf{1}_{\fn \ne \emptyset}\cdot (\fn^{+}, (q-1)*(r, \fm)) +  (\fn, 1, (q-1)*(r,\fm)).$$

We claim that in Eq. \eqref{eq: trivial induced relation}, $\widetilde{\ft}$ cannot appear among
$$\mathbf{1}_{\fn_{\ft} \ne \emptyset}\cdot (\fn_{\ft}^{+}, (q-1)*(r_{\ft}, \fm_{\ft})) + (\fn_{\ft}, 1, (q-1)*(r_{\ft},\fm_{\ft})$$
for some $\ft>\fs$ and $\ft \notin \mathcal J_w$. In fact, since $\ft>\fs$, the hypothesis implies that $\ft \ge \widetilde{\ft}$. But all tuples appearing in the expression
$$\mathbf{1}_{\fn_{\ft} \ne \emptyset}\cdot (\fn_{\ft}^{+}, (q-1)*(r_{\ft}, \fm_{\ft})) + (\fn_{\ft}, 1, (q-1)*(r_{\ft},\fm_{\ft})$$
are strictly greater than $\ft$, thus strictly greater than $\widetilde{\ft}$.

Therefore, the coefficient of $\widetilde{\ft}$ is nonzero. This contradicts the fact that Eq. \eqref{eq: trivial induced relation} is trivial.
\end{proof}

\subsection{Proof of Theorem \ref{thm: small weights}: \textsf{Type 1} and \textsf{Type 2}} \ppar

\begin{proposition} \label{prop: small weights type 1}
With the above notation, the tuple $\fs=s(\mathcal R)$ is neither of \textsf{Type~1} nor \textsf{Type~2}.
\end{proposition}

\begin{proof}
Assume that $\fs=s(\mathcal R)$ is of \textsf{Type 1} or \textsf{Type 2}. Thus $\fm=\emptyset $ and $\fs = (\fn, q+r)$ with a non-empty tuple $\fn$  consisting of integers $1, \dots, q$, and $r\ge0$.

We put
$$\widetilde{\ft}= (\fn^+, q+r-1).$$
Then for any tuple $\ft$, we note that  $\fs < \ft$ implies $\widetilde{\ft} \le \ft$. Also, the corresponding Eq.~\eqref{eq: revisted expression 1} is
  \begin{align*}
    &\Li(\fs) + \textbf{1}_{r\ne0} \Li(\fn, q, r) \\
    & + D_1 \Li(\widetilde{\ft}) + D_1 \Li(\fn^+, q-1, r)+ D_1 (\fn^+, r,q-1)
    + D_1 \Li(\fn, 1, (q-1)*(r)) =0.
  \end{align*}
By Lemma~\ref{lemma: special case}, we get a contradiction.
\end{proof}

\subsection{Proof of Theorem \ref{thm: small weights}: \textsf{Type 0} and \textsf{Type 3}} \ppar

\begin{proposition}\label{prop: small weights type 3}
Suppose that the tuple $\fs=s(\mathcal R)$ is of \textsf{Type 0} or \textsf{Type 3}. Then
$$\fs = (q+2, q-1),$$
which is of \textsf{Type 3}.
\end{proposition}

\begin{proof}
Since $\fs$ is of \textsf{Type 0} or \textsf{Type 3},   $\fn=\emptyset$ and we express $\fs = (q+r, \fm)$ with $r\ge1$ and $\fm = (m_1, \dots, m_\ell)$ a tuple of positive integers which may be empty.

Recall that by Proposition~\ref{prop: trivial induced relation}, the $\Fq$-linear relation $\Ind(\mathcal R)$ given by
\begin{align*}
\begin{split}
  0=& \Li(1, (q-1)*(r,\fm)) \\
  &+ \sum_{\ft>\fs, \, \ft \notin \mathcal J_w} a_{\ft} \left (\mathbf{1}_{\fn_{\ft} \ne \emptyset}\cdot \Li(\fn_{\ft}^{+}, (q-1)*(r_{\ft}, \fm_{\ft})) + \Li(\fn_{\ft}, 1, (q-1)*(r_{\ft},\fm_{\ft})) \right).\end{split}
\end{align*}
is trivial.

Suppose that $\fs \ne (q+2, q-1)$. We will choose a tuple $\widetilde{\ft}$  which appears among
\begin{align*}
(1, (q-1)*(r,\fm)).
\end{align*}
Then we show that in the previous equation of $\Ind(\mathcal R)$, the coefficient of $\widetilde{\ft}$ is nonzero.

We distinguish three cases.

\medskip
\noindent
    \textbf{Case 3.1: $r>1$.}

    We choose
    $$\widetilde{\ft} = (1, q+r-1, \fm).$$
    We note that $\widetilde{\ft}$ occurs as the first term in $(1, (q-1)*(r,\fm))$, and also
      $\widetilde{\ft}\ne(1, r, m_1, \dots, m_{i-1}, m_i + q-1, m_{i+1}, \dots, m_\ell)$ for any of $i$ as their second entries $q+r-1$ and $r$ differ.

\medskip
\noindent
      \textbf{Claim (A).} We claim that $\widetilde{\ft}=(1, q+r-1, \fm)$ does not occur in the expansion~of
      \begin{align*}
      & \mathbf{1}_{\fn_\ft\ne\emptyset}\cdot({\frak n}_{\ft}^+,(q-1)*(r_{\ft},\frak m_{\ft})), \quad \text{nor}\\
      &  (\frak n_{\ft},1,(q-1)*(r_{\ft},\frak m_{\ft})),
      \end{align*}
      for any $\ft=(\fn_{\ft}, q+r_{\ft}, \fm_{\ft})$ with $\ft > \fs$.

      \begin{proof}[Proof of the claim (A)]
      Recall that we have $r>1$.

      First, suppose that $\widetilde{\ft}=(1, q+r-1, \fm)$ appears among $\mathbf{1}_{\fn_\ft\ne\emptyset}\cdot({\frak n}_{\ft}^+,(q-1)*(r_{\ft},\frak m_{\ft}))$. Thus $\fn_{\ft} \ne 0$ and we need to find a prefix of the initial entries of $(1, q+r-1, \fm)$ to be $\fn_{\ft}^+$ for some non-empty $\fn_{\ft}$. As $\fn_{\ft}$ consists of positive integers less than or equal to~$q$, it follows that ${\frak n}_{\ft}^+=(1,q+r-1)$. Thus $r=2$ and ${\frak n}_{\ft}=(1,q)$. We get $\fs = (q+2, \fm)$, $\widetilde{\ft} = (1, q+1, \fm)$ and $\ft=(1,q,q+r_{\ft},\frak m_{\ft})$. Then $\widetilde{\ft} = (1, q+1, \fm)$ is among $(1,q+1,(q-1)*(r_{\ft},\frak m_{\ft}))$. As $w \leq 2q+1$, it implies that $(r_{\ft},\frak m_{\ft})=\emptyset$. Thus we get $\fs = (q+2, q-1)$, which is a contradiction.

      Next, suppose that $\widetilde{\ft}= (1, q+r-1, \fm)$ appears among the other part $({\frak n}_{\ft},1,(q-1)*(r_{\ft},\frak m_{\ft}))$. Thus the depth equality implies
      \[\depth(\fn_{\ft}) + \depth(\fm_{\ft})+2 \le \depth(\widetilde{\ft})\le \depth(\fn_{\ft}) + \depth(\fm_{\ft})+3.\]
      But
      \begin{align*}
      & \depth(\widetilde{\ft}) = \depth(\fs)+1, \\
      & \depth(\fn_{\ft}) + \depth(\fm_{\ft})+1 = \depth(\ft).
      \end{align*}
      Therefore, we conclude that
      \[\depth(\fs) \ge \depth(\ft).\]
      Since $\ft> \fs$, we have $\depth(\fs) = \depth(\ft)$ and $\ft \succ \fs$ where $\succ$ is the lexicographical order. In particular, $\widetilde{\ft}= (1, q+r-1, \fm)$ appears among either
      \[(\fn_{\ft}, 1, q+r_{\ft} -1, \fm_{\ft}) \]
      or
      \[(\fn_{\ft}, 1, r_{\ft}, m^{\ft}_1, \dots, m^{\ft}_{i-1},
      m^{\ft}_i + q-1, m^{\ft}_{i+1}, \dots, m^{\ft}_{\ell_{\ft}}
       )\] for some $i = 1, \dots, \ell_{\ft}$, where we write $\fm_{\ft} = (m^{\ft}_1, \dots, m^{\ft}_{\ell_{\ft}})$.

      If it were the first case, i.e.,  if
      \[ (1, q+r-1, \fm) = (\fn_{\ft}, 1, q+ r_{\ft}-1, \fm_{\ft}),\]
      then it follows that $\fn_{\ft} = \emptyset$, $r = r_{\ft}$, and $\fm = \fm_{\ft}$. This implies $\ft = \fs$, which contradicts the assumption $\ft > \fs$.

      If it were the second case, i.e., if
      \[(1, q+r-1, \fm) = (\fn_{\ft}, 1, r_{\ft}, m^{\ft}_1, \dots, m^{\ft}_{i-1},
      m^{\ft}_i + q-1, m^{\ft}_{i+1}, \dots, m^{\ft}_{\ell_{\ft}}
       )\]
       for some $i = 1, \dots, \ell_{\ft}$, then it follows that $\fn_{\ft}=\emptyset$, $q+r-1=r_{\ft}$ and $\fm=(m^{\ft}_1, \dots, m^{\ft}_{i-1}, m^{\ft}_i + q-1, m^{\ft}_{i+1}, \dots, m^{\ft}_{\ell_{\ft}})$. Since $r>1$ and $r_{\ft} \leq q+1$, it forces $r=2$ and $r_{\ft}=q+1$. Now recalling that $r_{\ft}+w(\fm_{\ft}) \leq q+1$, it follows that $\fm_{\ft}=\emptyset$. Therefore, $\fm=(q-1)$. Hence $\fs=(q+2,q-1)$, which is a contradiction.

       This proves the claim (A).
      \end{proof}

	 \noindent
      \textbf{Case 3.2: $r=1$ and $\fm=(m_1,\dots,m_\ell) \ne (\{1\}^\ell)$.}

      Let $1 \le k \le \ell$ be the smallest index such that $m_k > 1$. We can write
      \[\fs = (q+1, \{1\}^{k-1}, m_k,  \dots, m_{\ell}).\]
      We put
      \begin{align*}
      \widetilde{\ft} & = \left(\{1\}^{k+1}, m_k + q-1, m_{k+1}, \dots, m_\ell\right) .\end{align*}
      Note that $\widetilde{\ft} \ne (1, q+r-1, \fm)$ by comparing the second entry.

	 \medskip
	 \noindent
      \textbf{Claim (B).} We claim that $\widetilde{\ft}=\left(\{1\}^{k+1}, m_k + q-1, m_{k+1}, \dots, m_\ell\right) $ is not occurring in the expansion of
      \begin{align*}
      & \mathbf{1}_{\fn_\ft\ne\emptyset}\cdot({\frak n}_{\ft}^+,(q-1)*(r_{\ft},\frak m_{\ft})), \quad \text{nor}\\
      &  (\frak n_{\ft},1,(q-1)*(r_{\ft},\frak m_{\ft})),
      \end{align*}
      for any $\ft=(\fn_{\ft}, q+r_{\ft}, \fm_{\ft})$ with $\ft > \fs$.
      \begin{proof}[Proof of the claim (B)]
        We note that $\depth(\widetilde{\ft}) = \depth(\fs)+1$.

        Assume that $\widetilde{\ft}=\left(\{1\}^{k+1}, m_k + q-1, m_{k+1}, \dots, m_\ell\right)$ is found among $\mathbf{1}_{\fn_\ft\ne\emptyset}\cdot ({\frak n}_{\ft}^+,(q-1)*(r_{\ft},\frak m_{\ft}))$. In particular, $\fn_\ft \ne \emptyset$. Thus either $\widetilde{\ft} = (\fn_{\ft}^+, q+r_{\ft}-1, \fm_{\ft})$ or $\widetilde{\ft} = (\fn_{\ft}^+, r_{\ft}, m^{\ft}_1, \dots, m^{\ft}_{i-1}, q+m^{\ft}_i-1, m^{\ft}_{i+1}, \dots, m^{\ft}_{\ell_\ft})$.

\begin{enumerate}
\item Since $m_k>1$ and $w(\widetilde{\ft}) \leq 2q+1$, there is only one entry in $\widetilde{\ft}$, which is at least $q+1$, all entries to its left are $1$ and all entries to its right are at most $q-2$. Therefore, $(\fn_{\ft}^+, q+r_{\ft}-1, \fm_{\ft})$ cannot be a candidate as $\fn_{\ft}^+$ is either empty or consists of at least one entry greater than~$1$.

\item The remaining candidate is
        \begin{align*}
          & \left(\{1\}^{k+1}, m_k + q-1, m_{k+1}, \dots, m_\ell\right) \\
          &= (\fn_{\ft}^+, r_{\ft}, m^{\ft}_1, \dots, m^{\ft}_{i-1}, q+m^{\ft}_i-1, m^{\ft}_{i+1}, \dots, m^{\ft}_{\ell_\ft}).
        \end{align*}
        Therefore, $\fn_{\ft}= \emptyset$, $r_{\ft} = 1$, $\fm_{\ft} =
        (\{1\}^{k-1}, m_k, m_{k+1}, \dots, m_{\ell})$, that is,
        \[ \ft = (q+1, \{1\}^{k-1}, m_k, \dots, m_{\ell})= \fs.\]
        This contradicts to $\fn_\ft \ne \emptyset$.
\end{enumerate}

        Now we assume the remaining possibility, i.e., $\widetilde{\ft}=\left(\{1\}^{k+1}, m_k + q-1, m_{k+1}, \dots, m_\ell\right)$ is found
        among $({\frak n}_{\ft},1,(q-1)*(r_{\ft},\frak m_{\ft}))$. As $\depth(\widetilde{\ft}) = \depth(\fs)+1 \leq \depth(\ft)+1$, it implies that either 
        $\widetilde{\ft} = (\fn_{\ft},1, q+r_{\ft}-1, \fm_{\ft})$ or $\widetilde{\ft} = (\fn_{\ft},1, r_{\ft}, m^{\ft}_1, \dots, m^{\ft}_{i-1}, q+m^{\ft}_i-1, m^{\ft}_{i+1}, \dots, m^{\ft}_{\ell_\ft})$.

We note that 
        \begin{align*}
          \depth(\fn_{\ft},1, q+r_{\ft}-1, \fm_{\ft}) = \depth(\ft)+1,
        \end{align*}
        and similarly,
        \[ \depth(\fn_{\ft},1, r_{\ft}, m^{\ft}_1, \dots, m^{\ft}_{i-1}, q+m^{\ft}_i-1, m^{\ft}_{i+1}, \dots, m^{\ft}_{\ell_\ft}) = \depth(\ft)+1.\]
        Therefore, in either case, we conclude that $\depth(\ft) = \depth(\fs)$. Since we assume that $\ft > \fs$, we conclude that $\ft \succ \fs$ in the lexicographic order.

As $\fs = (q+1, \{1\}^{k-1}, m_k,  \dots, m_{\ell}) < \ft$, it follows that $\fn_{\ft}=\emptyset$. Thus $$r_{\ft} = 1, \quad \ft=(q+1,\fm_{\ft}) \quad \text{ and }$$
        $$\left(\{1\}^{k+1}, m_k + q-1, m_{k+1}, \dots, m_\ell\right) \\
         = (1,1,m^{\ft}_1, \dots, m^{\ft}_{i-1}, q+m^{\ft}_i-1, m^{\ft}_{i+1}, \dots, m^{\ft}_{\ell_\ft}).$$
         The weight equals $w \leq 2q+1$. We also note that $\depth(\fs)>2$ as otherwise $\fs=\ft$. It follows that $m^{\ft}_i < q$ for all $i$. As $m_k>1$, it implies that $m_k+q-1>q$. We deduce that $\fm_{\ft} =
        (\{1\}^{k-1}, m_k, m_{k+1}, \dots, m_{\ell})$, that is,
        \[ \ft = (q+1, \{1\}^{k-1}, m_k, \dots, m_{\ell})= \fs,\]
which is a contradiction. This completes the proof of the claim (B).
\end{proof}

\medskip
\noindent
      \textbf{Case 3.3: $r=1$ and $\fm = (\{1\}^{\ell})$.}

      We get $\ell = w-q-1$ and $\fs = (q+1, \{1\}^{w-q-1})$. We put
	 $$\widetilde{\ft} = (1, r, m_1, \dots, m_{\ell-1}, m_{\ell} + q-1) = (\{1\}^{w-q}, q).$$
	
\medskip
\noindent
      \textbf{Claim (C)}. We claim that $\widetilde{\ft} = (\{1\}^{w-q}, q)$ is not occurring in the expansion of
      \begin{align*}
      & \mathbf{1}_{\fn_\ft\ne\emptyset}\cdot({\frak n}_{\ft}^+,(q-1)*(r_{\ft},\frak m_{\ft})), \quad \text{nor}\\
      &  (\frak n_{\ft},1,(q-1)*(r_{\ft},\frak m_{\ft})),
       \end{align*}
       for any $\ft=(\fn_{\ft}, q+r_{\ft}, \fm_{\ft})$ with $\ft > \fs$.

       \begin{proof}[Proof of the claim (C)]
        Note that $\fs < \ft$ implies $\widetilde{\ft} \le \ft$. Therefore by Lemma~\ref{lemma: special case}, the claim (C) is true.\end{proof}

       We continue the proof of Proposition \ref{prop: small weights type 3}. From Cases 3.1, 3.2 and 3.3, there always exists  a tuple $\widetilde{\ft}$  which appears in Eq. \eqref{eq: revisted expression 1} among  the $D_1$-level tuples, i.e., among
$$\mathbf{1}_{\fn \ne \emptyset}\cdot (\fn^{+}, (q-1)*(r, \fm)) +  (\fn, 1, (q-1)*(r,\fm))$$
such that in the corresponding equation \eqref{eq: trivial induced relation}, the coefficient of $\widetilde{\ft}$ is nonzero. This contradicts the fact that Eq. \eqref{eq: trivial induced relation} is trivial.
\end{proof}

\subsection{Proof of Theorem \ref{thm: small weights}: \textsf{Type 4}} \ppar

\begin{proposition} \label{prop: small weights type 4}
We keep the previous notation. Then the tuple $\fs=s(\mathcal R)$ cannot be of \textsf{Type 4}.
\end{proposition}

\begin{proof}
Suppose that $\fs$ is of \textsf{Type 4}. Then we express $\fs = (\fn, q+r, \fm)$ with $r\ge1$, $\fn$ a non-empty tuple consisting of integers $1, \dots, q$, and $\fm = (m_1, \dots, m_\ell)$ a non-empty tuple of positive integers. Then Eq. \eqref{eq: revisted expression 1} is
  \begin{align*}
  \begin{split}
  &\Li(\fn, q+r, \fm) + \mathbf{1}_{r\ne0}\cdot \Li(\fn, q, r, \fm) \\
  &+ D_1 \mathbf{1}_{\fn \ne \emptyset}\cdot \Li(\fn^{+}, (q-1)*(r, \fm)) + D_1 \Li(\fn, 1, (q-1)*(r,\fm))=0.\end{split}
  \end{align*}

By Eq. \eqref{eq:bound for weight} we recall that $w(\fn)+w(r,\fm) \leq q+1$. In particular, $w(\fm)<q$.

Recall also that the $\Fq$-linear relation $\Ind(\mathcal R)$
\begin{align*}
\begin{split}
  0=& \mathbf{1}_{\fn \ne \emptyset}\cdot \Li(\fn^{+}, (q-1)*(r, \fm)) +  \Li(\fn, 1, (q-1)*(r,\fm)) \\
  &+ \sum_{\ft>\fs, \, \ft \notin \mathcal J_w} a_{\ft} \left (\mathbf{1}_{\fn_{\ft} \ne \emptyset}\cdot \Li(\fn_{\ft}^{+}, (q-1)*(r_{\ft}, \fm_{\ft})) + \Li(\fn_{\ft}, 1, (q-1)*(r_{\ft},\fm_{\ft})) \right).\end{split}
\end{align*}
is trivial.

We will choose a tuple $\widetilde{\ft}$  which appears in Eq. \eqref{eq: revisted expression 1} with the nonzero coefficient of $D_1$-degree one. Then we show that in the previous relation $\Ind(\mathcal R)$, the coefficient of $\widetilde{\ft}$ is nonzero.

We distinguish three cases.

\medskip
\noindent
    \textbf{Case 4.1: $r>1$.}

    We put
    $$\widetilde{\ft} = (\fn^+, q+r-1, \fm).$$
    Among Eq. \eqref{eq: revisted expression 1}, $\widetilde{\ft}$ occurs once in the expansion in $(\fn^+, (q-1)*(r, \fm))$.

\medskip
\noindent
      \textbf{Claim (A).} We claim that $\widetilde{\ft}=(\fn^+, q+r-1, \fm)$ is not occurring in the expansion~of
      \begin{align}
        \begin{split}& \mathbf{1}_{\fn_\ft\ne\emptyset}\cdot({\frak n}_{\ft}^+,(q-1)*(r_{\ft},\frak m_{\ft})), \quad \text{nor}\\
      &  (\frak n_{\ft},1,(q-1)*(r_{\ft},\frak m_{\ft}))
        \end{split} \label{eqn: 2q-1 revisited type 4A}
      \end{align}
      for any $\ft=(\fn_{\ft}, q+r_{\ft}, \fm_{\ft})$ with $\ft > \fs$.

      \begin{proof}[Proof of the claim (A)]
      We adapt the proof of Claim (A) as in Case 3.1. Recall that we have $r>1$.

      First, we assume that $\widetilde{\ft}=({\frak n}^+, q+r-1, \fm)$ is found among $\mathbf{1}_{\fn_\ft\ne\emptyset}\cdot({\frak n}_{\ft}^+,(q-1)*(r_{\ft},\frak m_{\ft}))$. In particular, ${\frak n}_{\ft} \neq \emptyset$. Recall that $\fs<\ft$. By comparing the depth, we deduce that either $\widetilde{\ft} = (\fn_{\ft}^+, q+r_{\ft}-1, \fm_{\ft})$ or $\widetilde{\ft} = (\fn_{\ft}^+, r_{\ft}, m^{\ft}_1, \dots, m^{\ft}_{i-1}, q+m^{\ft}_i-1, m^{\ft}_{i+1}, \dots, m^{\ft}_{\ell_\ft})$.
      \begin{itemize}
      \item Suppose that $({\frak n}^+, q+r-1, \fm)=\widetilde{\ft} = (\fn_{\ft}^+, q+r_{\ft}-1, \fm_{\ft})$. As the depth is at least $3$, on the LHS, we have a unique entry $q+r-1>q$ and $w(\fm)<q$. It follows that on the RHS, we must have a unique entry $>q$. As $w(\fm)<q$, the only possibility for $\fs<\ft$ is when $r=2$, $(\fn^+,q+1)=\fn_{\ft}^+$, $(r_{\ft},\fm_{\ft})=\emptyset$ and $\fm=(q-1)$. We get $\fs=(1,q+2,q-1)$ which is of weight $2q+2$, thus we obtain a contradiction.

      \item Otherwise, for some $i$, we have $({\frak n}^+, q+r-1, \fm)= (\fn_{\ft}^+, r_{\ft}, m^{\ft}_1, \dots, m^{\ft}_{i-1}, q+m^{\ft}_i-1, m^{\ft}_{i+1}, \dots, m^{\ft}_{\ell_\ft})$. Thus $\depth(\fs)=\depth(\ft)$. As $\fs<\ft$, it follows that $\fs \prec \ft$ in the lexicographical order. On the LHS, we have a unique entry $q+r-1>q$ and $w(\fm)<q$. Thus there exists $i<\ell<\ell_{\ft}$ such that
      $$({\frak n}^+, q+r-1)=(\fn_{\ft}^+, r_{\ft}, m^{\ft}_1, \dots, m^{\ft}_{i-1}, q+m^{\ft}_i-1, m^{\ft}_{i+1}, \dots, m^{\ft}_{\ell}).$$
      Since $\fs=(\fn,q+r,\fm)<\ft=(\fn_{\ft},q+r_{\ft},\fm_{\ft})$, we get a contradiction.
      \end{itemize}

      Next, we assume that $\widetilde{\ft}= ({\frak n}^+, q+r-1, \fm)$ appears among $({\frak n}_{\ft},1,(q-1)*(r_{\ft},\frak m_{\ft}))$. By the depth equality, we get
      \[\depth(\fn_{\ft}) + \depth(\fm_{\ft})+2 \le \depth(\widetilde{\ft})\le \depth(\fn_{\ft}) + \depth(\fm_{\ft})+3.\]
      But
      \begin{align*}
      \depth(\fs) &= \depth(\widetilde{\ft}), \\
      \depth(\ft) &=\depth(\fn_{\ft}) + \depth(\fm_{\ft})+1.
      \end{align*}
      Therefore,
      \[\depth(\fs) > \depth(\ft),\] which is a contradiction.

       This proves the claim (A).
      \end{proof}

\medskip
\noindent
      \textbf{Case 4.2: $r=1$ and $\fm=(m_1,\dots,m_\ell) \ne (\{1\}^\ell)$.}

      Let $1 \le k \le \ell$ be the smallest index such that $m_k > 1$. We can write that
      \[\fs = (\fn, q+1, \{1\}^{k-1}, m_k,  \dots, m_{\ell}).\]
      By Eq. \eqref{eq:bound for weight}, we know that $w(\fn)+w(r,\fm) \leq q+1$. As $\fn \neq \emptyset$, $m_k>1$ and $r>1$, it follows that $w(\fn)<q$ and $1<w(\fm)<q$.

      We put
      \begin{align*}
      \widetilde{\ft} =& \left(\fn^+, 1, m_1, \dots, m_{k-1}, m_k + q-1, m_{k+1}, \dots, m_\ell\right)\\
      = & \left(\fn^+, \{1\}^{k}, m_k + q-1, m_{k+1},\dots, m_\ell\right).\end{align*}
      We note that $\depth(\widetilde{\ft}) = \depth(\fn) + \depth(\fm) + 1 = \depth(\fs)$.

      Among Eq. \eqref{eq: revisted expression 1}, $\widetilde{\ft}$ occurs once in the expansion in $(\fn^+, (q-1)*(r, \fm))$. Also, it is not occurring in $(\fn, 1, (q-1)*(r,\fm))$, which is clear by the depth comparison.

\medskip
\noindent
      \textbf{Claim (B).} We claim that $\widetilde{\ft}$ is not occurring in the expansion of
      \begin{align*}
      & \mathbf{1}_{\fn_\ft\ne\emptyset}\cdot({\frak n}_{\ft}^+,(q-1)*(r_{\ft},\frak m_{\ft})), \quad \text{nor}\\
      &  (\frak n_{\ft},1,(q-1)*(r_{\ft},\frak m_{\ft})),
      \end{align*}
      for any $\ft=(\fn_{\ft}, q+r_{\ft}, \fm_{\ft})$ with $\ft > \fs$.

      \begin{proof}[Proof of the claim (B)]
        By the depth comparison, we can narrow down the possible candidates of $\widetilde{\ft}$. We are to find $\widetilde{\ft}$ among one of the terms in $({\frak n}_{\ft}^+,(q-1)*(r_{\ft},\frak m_{\ft}))$ with condition that $\fn_{\ft}$ is non-empty. We also assert that $\depth(\widetilde{\ft})=\depth(\ft) = \depth(\fs)$.

        First we suppose that
        \begin{align*}
          \widetilde{\ft} = \left(\fn^+, \{1\}^{k}, m_k + q-1, m_{k+1},\dots, m_\ell\right) = (\fn_{\ft}^+, q+r_{\ft}-1, \fm_{\ft})
        \end{align*}
        for some $\fn_\ft \ne \emptyset$. We list the following cases:
        \begin{itemize}
          \item If $\depth(\fn_\ft) < \depth(\fn)$, then $\fn_{\ft}^+$ is a proper prefix of $\fn^+$. In this case, we have $\ft< \fs$, which is a contradiction.
          \item If $\depth(\fn_\ft) = \depth(\fn)$, then $\fn_\ft = \fn$ and $1=q+r_{\ft}-1$, which is a contradiction.
          \item If $\depth(\fn)<\depth(\fn_\ft) \leq \depth(\fn,\{1\}^{k})$, then the last entry of $\fn_\ft^+$ is greater than $1$, while the entry at the same index in $\widetilde{\ft}$ is $1$. We get a contradiction.
          \item If $\depth(\fn_\ft) > \depth(\fn,\{1\}^{k})$, then it forces $m_k=2$,
         $$\left(\fn^+, \{1\}^{k}, q+1\right) = \fn_{\ft}^+, \quad \text{ and } \quad 
         \left(m_{k+1},\dots, m_\ell\right) = (q+r_{\ft}-1, \fm_{\ft}).$$
         Thus $w(\fm)\geq 2+(q-1)=q+1$, which implies $w(\fs)>2q+1$ and we obtain a contradiction.
        \end{itemize}

         Next we suppose that for some $i$,
        \begin{align*}
          \widetilde{\ft} &= \left(\fn^+, \{1\}^{k}, m_k + q-1, m_{k+1},\dots, m_\ell\right) \\
          &= (\fn_{\ft}^+, r_\ft, m^\ft_1, \dots, m^\ft_{i-1}, q+m^\ft_i-1, m^\ft_{i+1}, \dots, m^\ft_{\ell_\ft})
        \end{align*}
        We list the following cases:
        \begin{itemize}
          \item If $\depth(\fn_\ft) < \depth(\fn)$, then $\fn_{\ft}^+$ is a proper prefix of $\fn^+$. In this case, we have $\ft< \fs$, which is a contradiction.
          \item If $\depth(\fn_\ft) = \depth(\fn)$, then
          \begin{align*}
             \fn_\ft = \fn, \quad r_\ft = 1, \quad \text{ and } \quad
          \end{align*}
          \begin{align*}
          \left(\{1\}^{k-1}, m_k + q-1, m_{k+1},\dots, m_\ell\right) = (m^\ft_1, \dots, m^\ft_{i-1}, q+m^\ft_i-1, m^\ft_{i+1}, \dots, m^\ft_{\ell_\ft})
          \end{align*}
          As $w(\fm)<q$, it implies that $\fm_\ft = \fm$, i.e., $\ft = \fs$. This contradicts the assumption $\ft > \fs$ in the (depth-lex) order.
          \item If $\depth(\fn)<\depth(\fn_\ft) \leq \depth(\fn,\{1\}^{k})$, then the last entry of $\fn_\ft^+$ is greater than $1$, while the entry at the same index in $\widetilde{\ft}$ is $1$. So we get a contradiction.
          \item If $\depth(\fn_\ft) > \depth(\fn,\{1\}^{k})$, then it forces $m_k=2$,
         $$\left(\fn^+, \{1\}^{k}, q+1\right) = \fn_{\ft}^+, \quad \text{ and } $$
         $$\left(m_{k+1},\dots, m_\ell\right) = (r_\ft, m^\ft_1, \dots, m^\ft_{i-1}, q+m^\ft_i-1, m^\ft_{i+1}, \dots, m^\ft_{\ell_\ft}).$$
         Thus $w(\fm)\geq 2+(q-1)=q+1$, which implies $w(\fs)>2q+1$ and we obtain a contradiction.
        \end{itemize}

        This proves the claim (B). 
        \end{proof}

\medskip
\noindent
        \textbf{Case 4.3: $r=1$ and $\fm = (\{1\}^{\ell})$.}

        In this case, we have $\fs = (\fn, q+1, \{1\}^{\ell})$. We put
        $$\widetilde{\ft} = \left(\fn^+, \{1\}^{\ell}, q\right).$$
	   Among Eq. \eqref{eq: revisted expression 1}, $\widetilde{\ft}$ occurs once in the expansion in $(\fn^+, (q-1)*(r, \fm))$. Also, it is not occurring in $(\fn, 1, (q-1)*(r,\fm))$.

\medskip
\noindent
      \textbf{Claim (C)}. We claim that $\widetilde{\ft} = \left(\fn^+, \{1\}^{\ell}, q\right)$ is not occurring in the expansion of
      \begin{align*}
        & \mathbf{1}_{\fn_\ft\ne\emptyset}\cdot({\frak n}_{\ft}^+,(q-1)*(r_{\ft},\frak m_{\ft})), \quad \text{nor}\\
        &  (\frak n_{\ft},1,(q-1)*(r_{\ft},\frak m_{\ft})),
       \end{align*}
       for any $\ft=(\fn_{\ft}, q+r_{\ft}, \fm_{\ft})$ with $\ft > \fs$.

       \begin{proof}[Proof of the claim (C)]
       We note that if $\fs < \ft$ then $\widetilde{\ft}\le \ft$. Thus, it follows from Lemma~\ref{lemma: special case}.
       \end{proof}

       We return to the proof of Proposition \ref{prop: small weights type 4}. From Cases 4.1, 4.2 and 4.3, there always exists  a tuple $\widetilde{\ft}$  which appears in Eq. \eqref{eq: revisted expression 1} among  the $D_1$-level tuples, i.e., among
$$\mathbf{1}_{\fn \ne \emptyset}\cdot (\fn^{+}, (q-1)*(r, \fm)) +  (\fn, 1, (q-1)*(r,\fm)),$$
such that in the corresponding equation \eqref{eq: trivial induced relation}, the coefficient of $\widetilde{\ft}$ is nonzero. This contradicts the fact that Eq. \eqref{eq: trivial induced relation} is trivial.
\end{proof}

We now prove the main theorem of this section.
\begin{proof}[Proof of Theorem \ref{thm: small weights}]
Theorem \ref{thm: small weights} follows immediately from Propositions \ref{prop: small weights type 1}, \ref{prop: small weights type 3} and \ref{prop: small weights type 4}.
\end{proof}


\section{Proofs of Theorems \ref{thm: weight 2q+1} and \ref{thm: weight up to 2q}} \label{sec: main results}

\subsection{Proof of Theorem \ref{thm: weight up to 2q}} \ppar

Let $w \in \mathbb N$ such that $w \leq 2q$. We consider the set of nontrivial $\Fq$-linear relations among CMPL's of weight $w$. Suppose that this set is non-empty. Then we denote by $\mathcal R$ the nontrivial $\Fq$-linear relation such that $s(\mathcal R)$ is the largest tuple. By Corollary~\ref{cor: small weights}, $s(\mathcal R)=(q+2,q-1)$ which is of weight $2q+1$ and we obtain a contradiction.

We have proved that the set of CMPL's of weight $w$ are linearly independent over $\Fq$. By Theorem \ref{thmx: new connection}, we conclude that the set of MZV's of weight $w$ are also linearly independent over $\Fq$ as desired.

\subsection{Proof of Theorem \ref{thm: weight 2q+1}} \ppar

Let $w=2q+1$. We consider the set of nontrivial $\Fq$-linear relations among CMPL's of weight $w$. Suppose that this set is non-empty. Then we denote by $\mathcal R$ the nontrivial $\Fq$-linear relation such that $s(\mathcal R)$ is the largest tuple. Then by Corollary~\ref{cor: small weights}, we conclude that $s(\mathcal R)=(q+2,q-1)$.

The existence and uniqueness of such $\mathcal R$ follows from Propositions \ref{prop: weight 2q+1 existence} and \ref{prop: weight 2q+1 uniqueness} below.

\begin{proposition} \label{prop: weight 2q+1 existence}
For all $q$, there exists an $\mathbb{F}_q$-linear relation $\mathcal R$ of CMPL's of weight $2q+1$ such that $s(\mathcal R)=(q+2,q-1)$:
\begin{align} \label{eq: relation 2q+1}
& \Li(q+2, q-1)+\Li(q,2,q-1) -\Li(1,q, q) \\
&-\Li(1,1, 2q-1)-\Li(1,1,q,q-1) -\Li(1,q-1, 1,q) \notag \\
& +\Li(q+1,1, q-1)+\Li(q,1,1,q-1) -\Li(1,1, q-1,q) =0. \notag
\end{align}

Further, this relation comes from a binary relation given by
\begin{align*}
& \Si_d(q+2, q-1)+\Si_d(q,2,q-1) -\Si_{d+1}(1,q, q) \\
&-\Si_{d+1}(1,1, 2q-1)-\Si_{d+1}(1,1,q,q-1) -\Si_{d+1}(1,q-1, 1,q) \\
& +\Si_d(q+1,1, q-1)+\Si_d(q,1,1,q-1) -\Si_{d+1}(1,1, q-1,q) =0.
\end{align*}
\end{proposition}

\begin{proof}
We use \S \ref{sec: key expression} for the modified key expression as binary relations. Thus we get
\begin{align*}
\Si_d(q+2, q-1)  = &-\Si_d(q,2,q-1) - D_1 \Si_{d+1}(1,(q-1)*(2, q-1)),\\
-\Si_{d+1}(1,q, q)  = & D_1 \Si_{d+1}(1,q+1,q-1) + D_1 \Si_{d+1}(1,q,1,q-1), \\
-\Si_{d+1}(1,1, 2q-1)  = &\Si_{d+1}(1,1,q,q-1) + D_1 \Si_{d+1}(1,2,(q-1)*(q-1)) \\
& + D_1 \Si_{d+1}(1,1,1,(q-1)*(q-1)), \\
-\Si_{d+1}(1,q-1, 1,q)  = &D_1 \Si_{d+1}(1,q-1,2,q-1) + D_1 \Si_{d+1}(1,q-1,1,1,q-1), \\
\Si_d(q+1,1, q-1)  = & -\Si_d(q,1,1,q-1) - D_1 \Si_{d+1}(1,(q-1)*(1,1,q-1)), \\
-\Si_{d+1}(1,1, q-1,q)  = & D_1 \Si_{d+1}(1,1,q,q-1)+D_1 \Si_{d+1}(1,1,q-1,1,q-1).
\end{align*}

We collect the tuples from the $D_1$-level tuples. The sum of those terms, removed the first entry $(1)$ and divided by $D_1$, is:
\begin{align*}
& -(q-1)*(2, q-1) + (q+1,q-1) + (q,1,q-1)  \\
& + (2,(q-1)*(q-1)) + (1,1,(q-1)*(q-1)) \\
& + (q-1,2,q-1) + (q-1,1,1,q-1) \\
&- (q-1)*(1,1,q-1) + (1,q,q-1)+(1,q-1,1,q-1).
\end{align*}
We claim that this formal sum equals $0$. In fact, we split it into two sums:
\begin{align*}
S_1 = & -(q-1)*(2, q-1) + (q+1,q-1) \\
& + (2,(q-1)*(q-1)) + (q-1,2,q-1),  \\
S_2  = & (q,1,q-1) + (1,1,(q-1)*(q-1)) + (q-1,1,1,q-1)  \\
& - (q-1)*(1,1,q-1) + (1,q,q-1)+(1,q-1,1,q-1).
\end{align*}

We observe that for $s \in \bN$ and tuple $\fs=(s_1,\fs_-)$,
\begin{align*}
  (s)*(s_1, \fs_-) = (s+s_1, \fs_-) + (s,s_1, \fs_-)+ (s_1, (s)*\fs_-).
\end{align*}
Applying this formula for $(q-1)*(2, q-1)$, $(q-1)*(1,1,q-1)$ and $(q-1)*(1,q-1)$ yields $S_1=S_2=0$. Thus taking the sum over $d \in \mathbb Z$ of the binary relation yields the $\Fq$-linear relation \eqref{eq: relation 2q+1} and we complete the proof.
\end{proof}

\begin{proposition} \label{prop: weight 2q+1 uniqueness}
There exists at most one $\mathbb{F}_q$-linear relation $\mathcal R$ of CMPL's of weight $2q+1$ such that $s(\mathcal R)=(q+2,q-1)$.
\end{proposition}

\begin{proof}
Suppose that there exist two distinct $\mathbb{F}_q$-linear relations $\mathcal R, \mathcal R'$ of CMPL's of weight $2q+1$ such that $s(\mathcal R)=s(\mathcal R')=(q+2,q-1)$. We put $\fs=(q+2,q-1)$.

Suppose that $\mathcal R$ is given by
\[ \Li(\fs)+\sum_{\fs<\fs_i} a_i \Li(\fs_i) =0, \]
and that $\mathcal R'$ is given by
\[ \Li(\fs)+\sum_{\fs<\fs_i'} a_i' \Li(\fs_i') =0. \]
We denote by $\mathcal R''$ the $\mathbb{F}_q$-linear relation
\[ \sum_{\fs<\fs_i} a_i \Li(\fs_i) -\sum_{\fs<\fs_i'} a_i' \Li(\fs_i')=0. \]
Then this is a nontrivial $\mathbb{F}_q$-linear relation satisfying $s(\mathcal R'')>\fs=(q+2,q-1)$, which is a contradiction by Corollary \ref{cor: small weights}.
\end{proof}

Finally, we show that there is no other $\mathbb{F}_q$-linear relations $\mathcal R$ of CMPL's of weight $2q+1$.

\begin{proposition} \label{prop: weight 2q+1 uniqueness 2}
Let $\mathcal R$ be a nontrivial $\mathbb{F}_q$-linear relation of CMPL's of weight $2q+1$. Then $s(\mathcal R)=(q+2,q-1)$.
\end{proposition}

\begin{proof}
Recall that $w=2q+1$. We consider the set of all nontrivial $\mathbb{F}_q$-linear relations of CMPL's of weight $2q+1$ such that $s(\mathcal R) \ne (q+2,q-1)$. Suppose that this set is non-empty. Then we choose a nontrivial relation $\mathcal R$ in this set such that $s(\mathcal R)$ is the largest one.

Now we consider the induced relation $\Ind(\mathcal R)$. We claim that it is nontrivial. Otherwise, by Theorem \ref{thm: small weights}, $s(\mathcal R)=(q+2,q-1)$ which is a contradiction.

Since $\Ind(\mathcal R)$ is nontrivial, we know that
$$s(\Ind(\mathcal R))>s(\mathcal R).$$
By the choice of $\mathcal R$, we deduce that $s(\Ind(\mathcal R))=(q+2,q-1)$.

Recall that $\Ind(\mathcal R)$  is given as in Eq. \eqref{eq: induced relation 2} by
\begin{align*}
\begin{split}
  0=& \mathbf{1}_{\fn \ne \emptyset}\cdot \Li(\fn^{+}, (q-1)*(r, \fm)) +  \Li(\fn, 1, (q-1)*(r,\fm)) \\
  &+ \sum_{\ft>\fs, \, \ft \notin \mathcal J_w} a_{\ft} \left (\mathbf{1}_{\fn_{\ft} \ne \emptyset}\cdot \Li(\fn_{\ft}^{+}, (q-1)*(r_{\ft}, \fm_{\ft})) + \Li(\fn_{\ft}, 1, (q-1)*(r_{\ft},\fm_{\ft})) \right).\end{split}
\end{align*}
Here $\fs=(\fn,q+r,\fm)$. We obtain a contradiction as $q+2$ cannot be a prefix of any tuples appearing in the previous expression.
\end{proof}

We now write down explicitly the unique $\mathbb{F}_q$-linear relation among MZV's of weight $2q+1$.

\begin{proposition} \label{prop: MZV linear dependence in 2q+1 case}
For all $q$, there exists a unique $\mathbb{F}_q$-linear relations among MZV's of weight $2q+1$:
\begin{align*}
  & \zeta_A(q+2, q-1)+2\zeta_A(3, 2q-2)  +\zeta_A(q, 2, q-1) - \zeta_A(1, q, q) \\
  &- \zeta_A(1, 1, 2q-1)- \zeta_A(1, q-1, 1, q) + \zeta_A(q+1, 1, q-1) + \zeta_A(2, q-1, 1, q-1) \\
  & +\zeta_A(2, q, q-1) + \zeta_A(2, 1, 2q-2) + \zeta_A(q, 1, 1, q-1) - \zeta_A(1, 1, q-1, q)=0.
\end{align*}
\end{proposition}

\begin{proof}
By Eq. \eqref{eq: relation 2q+1},
\begin{align*}
& \Li(q+2, q-1)+\Li(q,2,q-1) -\Li(1,q, q) \\
&-\Li(1,1, 2q-1)-\Li(1,1,q,q-1) -\Li(1,q-1, 1,q) \notag \\
& +\Li(q+1,1, q-1)+\Li(q,1,1,q-1) -\Li(1,1, q-1,q) =0. \notag
\end{align*}

We note that for all $d \in \mathbb Z$ and $1 \leq s \leq q$,
\begin{align*}
S_d(s)&=\Si_d(s).
\end{align*}
Thus we have to express only the terms $\Li(q+2, q-1), \Li(1,1, 2q-1),\Li(q+1, 1, q-1)$ as linear combinations of MZV's as in Theorem \ref{thm: new connection}.

By Eqs. \eqref{eq: sum MZV} and \eqref{eq:product depth1}, for all $d \in \mathbb Z$, we have
\begin{align*}
\Si_d(q+1)&=\Si_d(q-1) \Si_d(2)=S_d(q-1) S_d(2) \\
&=S_d(q+1)+S_d(2,q-1), \\
\Si_d(q+2)&=\Si_d(q) \Si_d(2)=S_d(q) S_d(2) \\
&=S_d(q+2)+2 S_d(3,q-1), \\
\Si_d(2q-1)&=\Si_d(q-1) \Si_d(q)=S_d(q-1) S_d(q) \\
&=S_d(2q-1)- S_d(q,q-1), \\
S_{<d}(2q-2)&=S_{<d}(q-1) S_{<d}(q-1).
\end{align*}
Thus
\begin{align*}
\Si_d(q+2, q-1)&=S_d(q+2, q-1)+2S_d(3, 2q-2), \\
\Si_d(q+1,1, q-1)&=S_d(q+1, 1, q-1)+S_d( 2, q-1, 1,q-1)\\
&\quad+S_d(2, q,q-1)+S_d(2, 1,2q-2), \\
\Si_d(1,1, 2q-1)&=S_d(1,1, 2q-1)-S_d(1,1,q,q-1).
\end{align*}

Putting these altogether yields the corollary.
\end{proof}

We return to the proof of Theorem \ref{thm: weight 2q+1}. From Propositions \ref{prop: weight 2q+1 existence}, \ref{prop: weight 2q+1 uniqueness} and \ref{prop: weight 2q+1 uniqueness 2}, it follows that there exists a unique and explicit linear relations over $\Fq$​ among CMPL's of weight $2q+1$. Theorem \ref{thm: weight 2q+1} follows from this result and Theorem \ref{thmx: new connection}. Further, the unique $\Fq$-linear relation among MZV's of weight $2q+1$ is given as in Proposition~\ref{prop: MZV linear dependence in 2q+1 case}.


\appendix

\section{BCP relations} \label{sec: BCP relations}

In the addendum to our work, we discuss findings beyond the threshold $w=2q+1$. In this section we construct $\Fq$-linear relations called BCP relations in \S \ref{sec: BCP relations} and use them to get all the $\Fq$-linear relations of weight $w<3q$ (see Theorem \ref{thm: relations from P, B, C}), extending Theorems \ref{thm: weight 2q+1} and \ref{thm: weight up to 2q}. We also present speculations for general $w$ in \S \ref{sec: speculations}. The proof of Theorem \ref{thm: relations from P, B, C} is given in \S \ref{sec: w<3q Part 1} and \S \ref{sec: w<3q Part 2}.

\subsection{Binary relations and the Todd's operators revisited} \ppar

We define the set $\mathfrak{LC}_w$ of linear combinations as follows. Let $d$ be an integer. A binary linear combination at level $d$ is an $\mathbb{F}_q(D_1)$-linear combination of $\Si_{<d}(\fs_i)$'s and $\Si_{<d+1}(\ft_i)$'s, where weights of $\fs_i$'s and $\ft_j$'s are all $w$. In other words, the following expression is said to be a binary linear combination of weight $w$: \[l(d) = \sum_{i}a_i \Si_{<d} (\fs_i) + \sum_j b_j \Si_{<d+1}(\ft_j), \quad a_i, b_j \in \mathbb{F}_q(D_1),\ w(\fs_i) =w(\ft_j) = w.\]

If $l\in \mathfrak{LC}_w$, then $l(d)$ is the realization of the binary linear combination at level $d$. For $l, l' \in \mathfrak{LC}_w$ and $c \in \mathbb{F}_q(D_1)$ we define $l + cl' \in \mathfrak{LC}_w$ as a linear combination, whose realization at level $d$ is $l(d) + cl'(d)$. If $l(d) = 0$ for all $d$, then we write $l=0$; it is the binary relations as we previously defined in \S \ref{sec: binary relations}. Also, if $l-l'=0$ is the binary relation, we denote $l=l'$.

Let $l$ be a binary linear combination given by
\[l(d) = \sum_{i}a_i \Si_{<d} (\fs_i) + \sum_j b_j \Si_{<d+1}(\ft_j).\]
We define $l$ to be a fixed binary linear combination when $b_j =0$ for all $j$. We define $l$ to be a binary $\mathbb{F}_q$-linear combination when $a_i$'s and $b_j$'s are elements of $\mathbb{F}_q$ for all $i$'s and $j$'s.

For example, from the fundamental relation
\[\Si_d(q) = - D_1 \Si_{d+1}(1, q-1),\] 
we obtain
\[\Si_{<d}(q) = - D_1 \Si_{<d+1}(1, q-1).\]

\begin{definition} \label{defn: A and B}
Let $A(d) = \Si_{<d}(q)$ and $B(d) = \Si_{<d+1}(1, q-1)$. 
\end{definition}
Then $A$ and $B$ are binary $\mathbb{F}_q$-linear combination, and $A$ is a fixed binary linear combination. We write $A = -D_1 \cdot B$.

\subsection{Operators $\mathcal{D}$ and $\mathcal{D} \circ \iota$} \label{sec: operator D} \ppar

Let $l \in \mathfrak{LC}_w$ given by \[l(d) = \sum_{i}a_i \Si_{<d} (\fs_i) + \sum_j b_j \Si_{<d+1}(\ft_j).\]
We define $\mathcal{D}_n(l)\colon \mathfrak{LC}_w \to\mathfrak{LC}_{w+n}$ to be
\begin{align*}
  \mathcal{D}_n(l)(d) := \sum_{\ell<d}\Si_\ell(n)\cdot l(\ell).
\end{align*}

Recall that $A$ is a linear combination given by $A(d) = \Si_{<d}(q)$ (see Definition \ref{defn: A and B}), then
$\Si_d(1) \cdot A(d) = \Si_d(1)\Si_{<d}(q) = \Si_d(1, q)$ thus
\begin{align*}
  \mathcal{D}_1(A)(d) = \Si_{<d}(1, q).
\end{align*}
We define $\mathcal{D}_{(n_1,\dots, n_r)} = \mathcal{D}_{n_1}\circ \dots \circ \mathcal{D}_{n_r}$.

We also define the binary linear combination $\iota(l)$ when a fixed binary linear combination $l$ is given. Namely, if
    \[l(d) = \sum_{i}a_i \Si_{<d} (\fs_i)\]
then we define $\iota(l)$ to be a binary linear combination with
    \[\iota(l)(d) := \sum_{i}a_i \Si_{<d+1} (\fs_i).\]
That is, it replaces $\Si_{<d}$ into $\Si_{<d+1}$ for a fixed binary linear combination. For example, $\iota(A)(d)= \Si_{<d+1}(q)$.

Let $R\in \mathfrak{BR}$ be a binary linear relation given by
\[ R(d): \quad\sum a_i \Si_{d}(\fs_i) + \sum b_j \Si_{d+1}(\ft_j)=0\]
It induces a binary linear combination $l \in \mathfrak{LC}$ with
\[ l(d) = \sum a_i \Si_{<d}(\fs_i) + \sum b_j \Si_{<d+1}(\ft_j).\]
The process to obtain $\mathcal{B}^*_n(R)(d)$ and $\mathcal{D}_n(l)(d)$ is exactly the same. Precisely, \[\mathcal{D}_n(\ell)=0\] will be written as
\[\sum_{\ell<d}\mathfrak{B}^*_n(R)(\ell).\]

Another building block of the recipe is $\mathcal{D}_n \circ \iota$. We mention that $(\mathcal{D}_n\circ\iota) \left(\Si_{<d}(\mathfrak{s})\right)$ is related to $\mathcal{C}_{\mathfrak{s}} \left(\Si_d(n)\right)$ as
\begin{align*}
  & \mathcal{C}_{\mathfrak{s}} \left(\Si_d(n)\right) = \Si_d(n) \Si_{<d+1}(\fs).
\end{align*}

\subsection{Operator $\mathcal{D}^*$} \label{sec: operator Dstar} \ppar

Let $l$ be a fixed linear combination given by
\[ l(d) = \sum a_i \Si_{<d}(\fs_i).\]
We recall that
    \[\iota(l)(d) = \sum_{i}a_i \Si_{<d+1} (\fs_i),\]
and then
\begin{align*}
(\mathcal{D}_q\circ \iota)(l)(d)& =\sum_{\ell<d}\Si_\ell(q)\cdot \iota(l)(\ell) \\
& =\sum_{\ell<d}\Si_\ell(q)\cdot \sum_{i}a_i \Si_{<\ell+1} (\fs_i).
\end{align*}

As $\Si_\ell(q)=-D_1 \Si_{\ell+1}(1,q-1)$, we simply define
\begin{align*}
(\mathcal{D}^*_q\circ \iota)(l)(d)& :=-D_1 \sum_{\ell<d} \Si_{\ell+1}(1,q-1)\cdot \iota(l)(\ell) \\
& =-D_1 \sum_{\ell<d} \Si_{\ell+1}(1,q-1) \cdot \sum_{i}a_i \Si_{<\ell+1} (\fs_i) \\
& =-D_1 \sum_{\ell<d} \sum_{i}a_i \Si_{\ell+1} (1,(q-1)*\fs_i)\\
& =-D_1 \sum_{i}a_i \Si_{<d+1} (1,(q-1)*\fs_i).
\end{align*}
Then $(\mathcal{D}^*_q\circ \iota)(l)$ is a binary linear combination.

To summarize, we have
\begin{align*}
(\mathcal{D}_q\circ \iota)=(\mathcal{D}^*_q\circ \iota)
\end{align*}
but we obtain two different linear combinations. The RHS has a $D_1$ factor.

\subsection{BCP relations} \ppar
\label{subsubsec: a machine P}

By Definition \ref{defn: A and B}, we recall $A(d)=\Si_d(q)$ and $B(d)={\Si_{d+1}(1,q-1)}$. 

We define the following fixed binary $\mathbb{F}_q$-linear relation of weight $2q + w(\fs)$, for a nonempty tuple $\fs  = (s_1, \fs_-)$ as
\begin{align*}
P(\fs): \quad \frac{1}{D_1}\mathcal{D}^*_q\circ \iota \circ \mathcal{D}_{\fs}(A) + \mathcal{D}_q\circ\iota \circ\mathcal{D}_\fs (B)=0.
\end{align*}

Note that
\begin{align*}\mathcal{D}_q^* \circ \iota \circ \mathcal{D}_\fs(A)(d) = - D_1 \Si_{<d+1}(1, (q-1)*(\fs, q)),
\end{align*}
and
\begin{align*}
  \mathcal{D}_\fs(B)(d) &= \Si_{<d}(\fs^+, q-1)+\Si_{<d}(\fs, 1, q-1), \\
  \mathcal{D}_q\circ\iota\circ\mathcal{D}_\fs(B)(d) &= \Si_{<d}(q+(\fs^+)_1,(\fs^+)_-, q-1)+\Si_{<d}(q,\fs^+, q-1)\\
  & \quad + \Si_{<d}(q+s_1, \fs_-,1, q-1)+\Si_{<d}(q,\fs, 1, q-1).
\end{align*}

Therefore, for a nonempty tuple $\fs = (s_1, \fs_-)$ we have
\begin{align*}
  \begin{split}
  P(\fs): \quad &\Si_{<d}(q+(\fs^+)_1,(\fs^+)_-, q-1)+\Si_{<d}(q,\fs^+, q-1)\\
  &+ \Si_{<d}(q+s_1, \fs_-,1, q-1)+\Si_{<d}(q,\fs, 1, q-1)\\
  & - \Si_{<d+1}(1, (q-1)*(\fs, q))=0\end{split}
\end{align*}
Its $\depthlex$-smallest tuple is $(q+(\fs^+)_1,(\fs^+)_-, q-1)$.

By abuse of notation, we also denote by $P(\fs)$ the binary relation induced by the above $\Fq$-linear relation: 
\begin{align}
  \label{eqn: P relations}
  \begin{split}
  P(\fs): \quad &\Si_{d}(q+(\fs^+)_1,(\fs^+)_-, q-1)+\Si_{d}(q,\fs^+, q-1)\\
  &+ \Si_{d}(q+s_1, \fs_-,1, q-1)+\Si_{d}(q,\fs, 1, q-1)\\
  & - \Si_{d+1}(1, (q-1)*(\fs, q))=0\end{split}
\end{align}

\begin{proposition} \label{prop: P1}
The binary relation in Prop. \ref{prop: weight 2q+1 existence} is given by $P(1)$.
\end{proposition}

\begin{proof}
We set $\fs=(1)$. Then the Proposition follows from the explicit formula \eqref{eqn: P relations}, 
\end{proof}

We give another example with $\fs=(1,1)$. Then $P(1,1)$ is given by
\begin{align*}
  & \Si_{<d}(q+1,2, q-1)+\Si_{<d}(q,1,2,q-1) \\
  &+\Si_{<d}(q+1,1,1, q-1) +\Si_{<d}(q,1,1,1,q-1)\\
  & -\Si_{<d+1}(1, (q-1)*(1,1,q))=0.
\end{align*}

\begin{definition} \label{defn: BCP relations}
For any tuples $\fm, \fn$ which may be empty, we get an $\Fq$-linear relation $\mathcal{B}^*_\fm \circ \mathcal{C}_\fn(P(\fs))$ which is called a BCP relation.
\end{definition}

We write down the expression of $\mathcal{B}^*_\fm \circ \mathcal{C}_\fn (P(\fs))$. 
By \eqref{eqn: P relations}, we get
\begin{align}
  \label{eqn: CP relations}
  \begin{split}
  \mathcal{C}_{\fn} (P(\fs)): \quad &\Si_{d}(q+(\fs^+)_1,(\fs^+)_-, q-1) \Si_{\leq d}(\fn)+\Si_{d}(q,\fs^+, q-1)\Si_{\leq d}(\fn) \\
  &+ \Si_{d}(q+s_1, \fs_-,1, q-1)\Si_{\leq d}(\fn) +\Si_{d}(q,\fs, 1, q-1)\Si_{\leq d}(\fn) \\
  & - \Si_{d+1}(1, (q-1)*(\fs, q))\Si_{<d+1}(\fn) =0.
  \end{split}
\end{align}
Then we can write down the expression of $\mathcal{B}^*_\fm \circ \mathcal{C}_\fn (P(\fs))$. 

In Sections \ref{sec: w<3q Part 1} and \ref{sec: w<3q Part 2} we prove the following theorem which extends Theorems \ref{thm: weight 2q+1} and \ref{thm: weight up to 2q}:

\begin{theorem} \label{thm: relations from P, B, C}
For $w<3q$,
\begin{enumerate}
\item The relations $\mathcal{B}^*_\fm \circ \mathcal{C}_\fn(P(\fs))$'s of weight $w$ are linearly independent over $\Fq$.

\item $P(\fs)$, $\mathcal{B}^*$, and $\mathcal{C}$ generates all possible $\mathbb{F}_q$-linear relations. That means all $\mathbb{F}_q$-linear relations of weight $w$ are linear combinations of $\mathcal{B}^*_\fm \circ \mathcal{C}_\fn(P(\fs))$'s over $\Fq$.
\end{enumerate}
\end{theorem}

\subsection{Some speculations} \ppar
\label{sec: speculations}

We report some speculations following Theorem \ref{thm: relations from P, B, C} and its proof which suggest that the set of BCP relations
\[\{\mathcal{B}^*(\fm)\circ \mathcal{C}(\fn) \left(P(\fs)\right) : w(\fm) + w(\fn) + w(\fs) = w - 2q, \fs \ne\emptyset\},\]
spans the vector space of all $\Fq$-linear relations.

The dimensions of the vector space spanned by BCP relations of weight $w$ are presented in the Table below (we wrote the Python and SageMath code by ourselves). 
\begin{table}[htbp] 
    \centering
  \caption{Number of all $\mathbb{F}_q$-linear relations of BCP types for $w > 2q$.}
\begin{tabular}{|c||c|c|c|c|c|c|c|c|c|c|c|c|c|} 
  \hline
$w-2q$& $1$ & $2$ & $3$ & $4$ & $5$&$6$& $7$ & $8$ & 9 & 10 & 11 & 12 \\
\hline\hline
$q=2$     & 1 & 4 & 13 & 36 & 92 & 222 & 515 & 1160 & 2555 & 5530 & 11804 & 24916\\
\hline
$q=3$    & 1 & 4 & 13 & 38 & 102 & 260 & 638& 1520&3539&8088&18203&40446 \\
\hline
$q=4$  & 1   & 4   &  13   & 38  &  104 & 270 & 676 & 1646 & 3920 & 9168 & 21123 & 48056 \\
\hline
$q=5$     & $1$   & $4$   & $13$   & $38$      & $104$   & $272$  \\
\hline
$q=7$     &  1  &  4  &  13   & 38  & 104   \\
\hline
\end{tabular}
\label{table: number of independent BCP relations}
\end{table}

By observing this table we give one possible formula to explain the numerical data. We define the following two-dimensional array. 
\begin{align*}
    \begin{matrix}
      \delta^{(1)}_1& \delta^{(1)}_2 &\delta^{(1)}_3 & \delta^{(1)}_4 & \delta^{(1)}_5 & \delta^{(1)}_6 & \delta^{(1)}_7 & \delta^{(1)}_8 & \delta^{(1)}_9 & \delta^{(1)}_{10} & \delta^{(1)}_{11} & \delta^{(1)}_{13} & \dots \\ 
      1 & 4 & 13 & 38 & 104 & 272 & 688 & 1696 & 4096 & 9728 & 22784 & 52736 & \dots \\
      \hline
\delta^{(2)}_1 & \delta^{(2)}_2 &\delta^{(2)}_3 & \delta^{(2)}_4 & \delta^{(2)}_5 & \delta^{(2)}_6 & \delta^{(2)}_7 & \delta^{(2)}_8 & \delta^{(2)}_9 & \dots \\
0 & 1&  6&  25&  88& 280 &832 & 2352 & 6400& \dots \\
\hline 
 \delta^{(3)}_1 & \delta^{(3)}_2 & \delta^{(3)}_3& \delta^{(3)}_4& \delta^{(3)}_5& \delta^{(3)}_6& \dots \\
 0 & 0 & 1 & 8 & 41 & 170& \dots \\
 \hline 
 \delta^{(4)}_1 & \delta^{(4)}_2 & \delta^{(4)}_3& \delta^{(4)}_4& \dots \\
 0 & 0 & 0 &1 & \dots \\
 \hline
 \vdots & \vdots & \vdots & 
    \end{matrix}
\end{align*}
\begin{itemize}
\item Initial data: $\delta^{(1)}_n = (n^2+5n+2) 2^{n-4}$ with $\delta^{(1)}_1 = 1$ is the number of all $BCP$'s of weight $w = 2q+n$. Also, $\delta^{(k)}_{1}=0$ for all $k>1$. For the sake of the convenience, let's put $\delta^{(k)}_n =0$ for all $n\le 0$ for all $k\ge1$.
\item Recurrence formula: $\delta^{(k)}_n = \delta^{(k-1)}_{n-1} + 2 \cdot \delta^{(k)}_{n-1}$
\item Generating function: For each $k\ge1$, the corresponding row have the generating function  $\sum_{n\ge1}\delta^{(k)}_n x^n= \frac{x^{k}(1-x)^2}{(1-2x)^{k+2}}$. Here $x^{k}$ factor from the numerator is for the initial $0$'s; when reindexed (to be indexed from the zeroth term), for each $k$'s, the sequence starting with $1$ has the generating function $\frac{(1-x)^2}{(1-2x)^{k+2}}$. 
\end{itemize}
Then for each finite $w\ge2q+1$, we expect that the dimension is given by the following finite sum: 
\[\dim_{\Fq}\langle \{BCPs\text{ of weight $w$}\}\rangle = \delta^{(1)}_{w-2q} - 2 \delta^{(2)}_{w-3q} + 3 \delta^{(3)}_{w-4q} - 4 \delta^{(4)}_{w-5q} + \dots.\]
This formula is supported by Theorem \ref{thm: relations from P, B, C} and the data presented in the above Table~\ref{table: number of independent BCP relations}.

Assume that this formula is true. Then 
\begin{align*}
  \dim_{\Fq}\langle \{BCPs\text{ of weight $w$}\}\rangle
   &= \sum_{j\ge 1}(-1)^{j+1}\, j\, \delta^{(j)}_{w-q-jq}.
\end{align*}
We expect the generating series:
\begin{align*}
  \sum_{w\ge 2q+1}\dim_{\Fq}\langle \{BCPs\text{ of weight $w$}\}\rangle x^w 
   &= \sum_{j\ge1}(-1)^{j+1} j\, x^{(j+1)q}\sum_{n\ge1}\delta^{(j)}_n x^n\\
   &= \frac{x^q(1-x)^2}{(1-2x)^2}\sum_{j\ge1}(-1)^{j+1} j
      \left(\frac{x^{q+1}}{1-2x}\right)^{j}\\
   &= \frac{x^q(1-x)^2}{(1-2x)^2}\cdot
      \frac{x^{q+1}(1-2x)}{\left(1-2x+x^{q+1}\right)^2}\\
   &= \frac{x^{2q+1}(1-x)^2}{(1-2x)\left(1-2x+x^{q+1}\right)^{2}}.
\end{align*}
In summary, expanding the generating function $\frac{x^{2q+1}(1-x)^2}{(1-2x)\left(1-2x+x^{q+1}\right)^{2}}$ as a Taylor series of $x$ at $x=0$ yields the sequence $\left(\dim_{\Fq}\langle \{BCPs\text{ of weight $w$}\}\rangle\right)_{w\ge1}$.


\section{Proof of Theorem \ref{thm: relations from P, B, C}, Part (1)} \label{sec: w<3q Part 1}

\subsection{Independence of CP relations} \ppar

We prove
\begin{proposition} \label{prop: strategy step 1}
Suppose that $w<3q$. Then $\mathcal{C}_\fn (P(\fs))$'s with $w(\fn)+w(\fs)=w-2q$ are linearly independent over $\Fq$.
\end{proposition}

The rest of this section is devoted to proving this proposition.

Suppose that we have an equality
\begin{equation} \label{eq: only C 1}
\mathcal{C}_\fn (P(\fs)) + \sum_i a_i \mathcal{C}_{\fn_i} (P(\fs_i))=0
\end{equation}
where
\begin{itemize}
\item $\fn, \fs$ such that $w(\fn)+w(\fs)=w-2q$,
\item $a_i \in \Fq$,
\item $\fn_i, \fs_i$ verifying $w(\fn_i)+w(\fs_i)=w-2q$ and $\depth(\fn_i) \leq \depth(\fn)$.
\end{itemize}
It is clear that $\depth(\fn)>0$.

Let $\fn=(n_1,\dots,n_k)$. We will show that the tuple
\[ (q+(\fs^+)_1,(\fs^+)_-, q-1+n_1, n_2, \dots, n_k)\]
has a non-zero coefficient in Eq. \eqref{eq: only C 1}. In fact, it follows from the lemmas below.

\begin{lemma} \label{lem: special tuple}
Recall that $\fn, \fs$ are non-empty tuples and $w(\fn)+w(\fs)=w-2q<q$. Then the tuple
\[ (q+(\fs^+)_1,(\fs^+)_-, q-1+n_1, n_2, \dots, n_k)\]
verifies the following properties:
\begin{enumerate}
\item All the entries are different from $q-1$.
\item There are exactly two entries which are strictly greater than $q-1$: the first entry $q+(\fs^+)_1$ and $q-1+n_1$.
\item The other entries are strictly smaller than $q-1$.
\item There are exactly $\depth(\fn)-1$ terms after the entry $q-1+n_1$ which is the last entry strictly greater than $q-1$.
\end{enumerate}
\end{lemma}

\begin{proof}
The proof is immediate.
\end{proof}

We note that Properties (1) and (4) will be very useful in the sequel.

\begin{lemma}
Let $(\fn_i, \fs_i) \neq (\fn, \fs)$ verifying $w(\fn_i)+w(\fs_i)=w-2q$ and $\depth(\fn_i) \leq \depth(\fn)$. Then the tuple
\[ (q+(\fs^+)_1,(\fs^+)_-, q-1+n_1, n_2, \dots, n_k)\]
has a zero coefficient in $\mathcal{C}_{\fn_i} (P(\fs_i))$.
\end{lemma}

\begin{proof}
By \eqref{eqn: CP relations},
\begin{align*}
  \begin{split}
  \mathcal{C}_{\fn_i} (P(\fs_i)): \quad &\Si_{d}(q+(\fs^+_i)_1,(\fs^+_i)_-, q-1) \Si_{\leq d}(\fn_i)+\Si_{d}(q,\fs^+_i, q-1)\Si_{\leq d}(\fn_i) \\
  &+ \Si_{d}(q+s_{i1}, \fs_{i-},1, q-1)\Si_{\leq d}(\fn_i) +\Si_{d}(q,\fs_i, 1, q-1)\Si_{\leq d}(\fn_i) \\
  & - \Si_{d+1}(1, (q-1)*(\fs_i, q))\Si_{<d+1}(\fn_i) =0.
  \end{split}
\end{align*}

Suppose that the tuple
\[ (q+(\fs^+)_1,(\fs^+)_-, q-1+n_1, n_2, \dots, n_k)\]
appears in the previous expression of $\mathcal{C}_{\fn_i} (P(\fs_i))$. By Lemma \ref{lem: special tuple}, we are looking for tuples satisfying Property (1) of Lemma \ref{lem: special tuple}. They will be of the form
\[ (\dots, q-1+\fn_{i,m}, \fn_{i,m+1}, \dots, \fn_{i,\ell}) \]
where $\fn_i=(\fn_{i,1},\dots,\fn_{i,\ell})$ and $1 \leq m \leq \ell$. As $\ell=\depth(\fn_i) \leq \depth(\fn)$, Property (4) forces $k=1$ and $\ell=\depth(\fn_i)=\depth(\fn)$. This proves $\fn = \fn_i$.

We claim that $\fs = \fs_i$. In fact, among the candidates in $\mathcal{C}_\fn(P(\fs_i))$, the tuple that starts with an element $>q$ are one of the following:
\begin{align*}
  (q + (\fs_i^+)_1 + n_1,\quad& ((\fs_i^+)_-, q-1)*\fn_-)& \\
  (q + (\fs_i^+)_1,\quad& ((\fs_i^+)_-, q-1)*\fn)& \\
  (q + n_1,\quad& (\fs_i^+, q-1)*\fn_-)& \\
  (q + s_{i1} + n_1,\quad& ((\fs_i)_-, 1, q-1)*\fn_-)& \\
  (q + s_{i1},\quad& ((\fs_i)_-, 1, q-1)*\fn,) & \\
  (q + n_1,\quad& (\fs_i, 1, q-1) * \fn_-).
\end{align*}
Note that in the expansion of $(\fa, q-1) * \fb$, the term $(q-1)$ from the first factor must be merged with any of $\fb$'s entries, not be written standalone, due to Property (1) of Lemma~\ref{lem: special tuple}. Therefore, the above candidates can be narrowed; we can get remove those that has $\fn_-$ as the right factor of the last $*$ product, and the remainings will be written as
\begin{align*}
  &(q + (\fs_i^+)_1, (\fs_i^+)_-, q-1+n_1, \fn_-),\\
  &(q + s_{i1}, (\fs_i)_-, 1, q-1 + n_1, \fn_-).\\
\end{align*}
If it were the first case, we obtain that $\fs_i = \fs$, which is a contradiction. If it were the second case, we have
\[ \fs^+ = (s_{i1}, (\fs_i)_-, 1),\]
which is again a contradiction since the last element of $\fs^+$ must be greater than $1$.
\end{proof}

\begin{lemma}
The tuple
\[ (q+(\fs^+)_1,(\fs^+)_-, q-1+n_1, n_2, \dots, n_k)\]
has a non-zero coefficient in $\mathcal{C}_\fn (P(\fs))$.
\end{lemma}

\begin{proof}
By \eqref{eqn: CP relations},
\begin{align*}
  \begin{split}
  \mathcal{C}_\fn (P(\fs)): \quad &\Si_{d}(q+(\fs^+)_1,(\fs^+)_-, q-1) \Si_{\leq d}(\fn)+\Si_{d}(q,\fs^+, q-1)\Si_{\leq d}(\fn) \\
  &+ \Si_{d}(q+s_1, \fs_-,1, q-1)\Si_{\leq d}(\fn) +\Si_{d}(q,\fs, 1, q-1)\Si_{\leq d}(\fn) \\
  & - \Si_{d+1}(1, (q-1)*(\fs, q))\Si_{<d+1}(\fn) =0.
  \end{split}
\end{align*}

By Lemma \ref{lem: special tuple}, we see that the tuple
\[ (q+(\fs^+)_1,(\fs^+)_-, q-1+n_1, n_2, \dots, n_k)\]
appears once in
\[ \Si_{d}(q+(\fs^+)_1,(\fs^+)_-, q-1) \Si_{\leq d}(\fn). \]
However, it does not appear in
\begin{align*}
  \begin{split}
  &\Si_{d}(q,\fs^+, q-1)\Si_{\leq d}(\fn) \\
  &+ \Si_{d}(q+s_1, \fs_-,1, q-1)\Si_{\leq d}(\fn) +\Si_{d}(q,\fs, 1, q-1)\Si_{\leq d}(\fn) \\
  & - \Si_{d+1}(1, (q-1)*(\fs, q))\Si_{<d+1}(\fn).
  \end{split}
\end{align*}
We get the lemma.
\end{proof}

\subsection{Independence of BCP relations} \ppar

\begin{proposition} \label{prop: lower bound} 
Suppose that $w<3q$. Then $\mathcal{B}^*_\fm \circ \mathcal{C}_\fn (P(\fs))$'s with $w(\fm)+w(\fn)+w(\fs)=w-2q$ are linearly independent over $\Fq$.
\end{proposition}

The rest of this section is devoted to proving this proposition.

For a tuple $\fn$, we write $\fn=(n_1,\dots,n_\ell)$. Recall that the initial tuple of $\fn$ denoted by $\Init(\fn)$ is defined to be the tuple $(n_1,\dots,n_k)$ where $n_1,\dots,n_k \leq q$ and $n_{k+1}>q$.

By \eqref{eqn: CP relations}, recall that
\begin{align*}
  \begin{split}
  \mathcal{C}_{\fn} (P(\fs)): \quad &\Si_{d}(q+(\fs^+)_1,(\fs^+)_-, q-1) \Si_{\leq d}(\fn)+\Si_{d}(q,\fs^+, q-1)\Si_{\leq d}(\fn) \\
  &+ \Si_{d}(q+s_1, \fs_-,1, q-1)\Si_{\leq d}(\fn) +\Si_{d}(q,\fs, 1, q-1)\Si_{\leq d}(\fn) \\
  & - \Si_{d+1}(1, (q-1)*(\fs, q))\Si_{<d+1}(\fn) =0.
  \end{split}
\end{align*}
Then we can write down the expression of $\mathcal{B}^*_\fm \circ \mathcal{C}_\fn (P(\fs))$. As $w(\fm)+w(\fn)+w(\fs)=w-2q<q$, it follows that

\begin{lemma} \label{lemma: initial tuple}
Any tuple appearing in $\mathcal{B}^*_\fm \circ \mathcal{C}_\fn (P(\fs))$ has the initial tuple which is greater than $\fm$.
\end{lemma}

\begin{proof}
We note that any tuple appearing in $\mathcal{B}^*_\fm \circ \mathcal{C}_\fn (P(\fs))$ has the prefix either $\fm$ or $\fm^+$. As $w(\fm)<q$, all the entries of $\fm_+$ are smaller than $q$. Thus we are done.
\end{proof}

Suppose that we have an equality
\begin{equation} \label{eq: B C}
\mathcal{B}^*_\fm \circ \mathcal{C}_\fn (P(\fs)) + \sum_i a_i \mathcal{B}^*_{\fm_i} \circ \mathcal{C}_{\fn_i} (P(\fs_i))=0
\end{equation}
where
\begin{itemize}
\item $\fm, \fn, \fs$ such that $w(\fm)+w(\fn)+w(\fs)=w-2q$,
\item $a_i \in \Fq$,
\item $(\fm_i,\fn_i, \fs_i) \neq (\fm,\fn,\fs)$ verifying $w(\fn_i)+w(\fs_i)=w-2q$ and either $\fm<\fm_i$ or ($\fm=\fm_i$ and $\depth(\fn_i) \leq \depth(\fn)$).
\end{itemize}

Then by Lemma \ref{lemma: initial tuple} and similar arguments as in the proof of Proposition \ref{prop: strategy step 1}, we show that  the tuple
\[ (\fm,q+(\fs^+)_1,(\fs^+)_-, q-1+n_1, n_2, \dots, n_k)\]
has a non-zero coefficient in Eq. \eqref{eq: B C}.


\section{Proof of Theorem \ref{thm: relations from P, B, C}, Part (2)} \label{sec: w<3q Part 2}

In what follows, a tuple $\fs$ is either empty or a sequence of integers in $\{1,2,\dots\}$. Let $\fs, \fs'$ be two tuples. Then we denote by $\fs + \fs'$ the tuple obtained by adding entry-wise.

\subsection{Auxiliary lemmas for arbitrary weights} \ppar

\subsubsection{The quantity $\alpha$} \ppar

Let $\fs$ be a tuple. If it can be expressed as 
\[ \fs=(\fm',(q-1)+\fn')=(\fm',q-1+n'_1,\dots,n'_k)  \] 
such that  $n_1' \geq 0$ and $1 \leq n'_2,\dots,n'_k < q-1$, then we define $\alpha(\fs):=\depth((q-1)+\fn')$. Otherwise, we define $\alpha(\fs):=-\infty$.

Let $R$ be an $\Fq$-linear relation of any weight $w$. We denote by $\mathcal{I}(R)$ the set of non-basis tuples $\fs$ in $R$. Let $\fs \in \mathcal I(R)$, we have the first expression
\begin{equation*}
\fs=(\fm,(q)+ \fn)
\end{equation*}
where $\fm$ and $\fn$ are tuples which may be empty. We need the second expression 
\[ \fs=(\fm',(q-1)+\fn')=(\fm',q-1+n'_1,\dots,n'_k)  \] 
such that $n_1' \geq 0$ and $1 \leq n'_2,\dots,n'_k < q-1$. We see that $\fm$ is a prefix of $\fm'$. Thus we get
\[ \fs=(\fm,(q)+\frak p,(q-1)+\fn'). \]
We emphasize that if $\frak p=\emptyset$, then $\fm=\fm'$ and $\fn'=(1)+\fn$. With this notation we define
\begin{align*}
\alpha(R) &:= \max_{\fs \in \mathcal I(R)} \alpha(\fs),
\end{align*}
and set
\begin{align*}
\mathcal S_{\max}(R) &:= \{ \fs \in \mathcal I(R): \alpha(\fs)=\alpha(R) \}.
\end{align*}
If $R=0$, then we set $\alpha(R):=-\infty$.

We prove some auxiliary lemmas. Let $\fs \in \mathcal I(R)$. Then we can write $\fs=(\fm,(q)+\fn)$ where $\fm=(m_1,\dots)$ and $\fn$ are tuples which may be empty and $m_i \leq q$ for all $i$. Then the tuples in $\Ind R$ coming from $\fs$ are given by
\begin{align*}
\begin{split}
\mathbf{1}_{\{\fm \ne \emptyset\}}\cdot (\fm^{+}, (q-1)*\fn) + (\fm, 1, (q-1)*\fn).
\end{split} 
\end{align*}

\begin{lemma} \label{lemma: bound for alpha}
We have $\alpha(R) \geq \alpha(\Ind R)$.
\end{lemma}

\begin{proof}
Let $\fs=(\fm,(q)+\fn)$ as above. Then in $\Ind R$, we get the tuples 
\begin{align*}
\begin{split}
\mathbf{1}_{\{\fm \ne \emptyset\}}\cdot (\fm^{+}, (q-1)*\fn) + (\fm, 1, (q-1)*\fn).
\end{split} 
\end{align*}
We see that
\begin{align*}
\alpha(\fm,(q)+\fn) = \alpha(\fm^{+}, (q-1)*\fn) \\
\alpha(\fm,(q)+\fn) = \alpha(\fm,1, (q-1)*\fn).
\end{align*}
For the RHS, we understand that $\alpha$  stands for the max of all the $\alpha(\fs')$ where $\fs'$ belongs to the RHS. The lemma follows.

We mention that the inequality comes from possible cancellations in $\Ind R$ and also from the fact that in the definition of $\alpha(\Ind R)$ we only consider the non basis tuples in $\Ind R$. 
\end{proof}

\subsubsection{The quantity $\depth^+$} \ppar

Let $\fs$ be a non basis tuple of any weight. Then we can write
\begin{equation*}
\fs=(\fm,(q)+ \fn)
\end{equation*}
where $\fm$ and $\fn$ are tuples which may be empty such that all the entries of $\fm$ are at most $q$. By convention, if $\fm=\emptyset$, then $\fm^+:=(1)$. We define \[\depth^+(\fs) = 
\begin{cases}
\depth(\fs) & \text{if } \fm \ne \emptyset, \\
\depth(\fs)+1 & \text{if } \fm = \emptyset.
\end{cases}\]
In other words, $\depth^+(\fs)=\depth(\fm^+, (q)+ \fn)$.

\begin{lemma} \label{lemma: depth plus}
Let $\fs\le \fs_0$ be two non basis tuples of the same weight. Suppose that $\fs\le \fs_0$ with respect to the $\depthlex$-order, then $\depth^+(\fs) \le \depth^+(\fs_0)$.
\end{lemma}

\begin{proof}
We write $\fs=(\fm,(q)+ \fn)$ and $\fs_0=(\fm_0,(q)+ \fn_0)$ as above. As $\fs\le \fs_0$, we get $\depth(\fs) \le \depth(\fs_0)$.

First, suppose that $\fm \neq \emptyset$, then 
\begin{align*}
\depth^+(\fs)=\depth(\fs) \leq \depth(\fs_0) \leq \depth^+(\fs_0).
\end{align*} 
Second, suppose that  $\fm = \emptyset$ and $\depth(\fs) < \depth(\fs_0)$, then
\begin{align*}
\depth^+(\fs)=\depth(\fs)+1 \leq \depth(\fs_0) \leq \depth^+(\fs_0).
\end{align*} 
Finally, we assume that  $\fm = \emptyset$ and $\depth(\fs) = \depth(\fs_0)$. Then $\fs \prec \fs_0$. It follows that $\fm_0=\emptyset$. We get
\begin{align*}
\depth^+(\fs)=\depth(\fs)+1 \leq \depth(\fs_0)+1 = \depth^+(\fs_0).
\end{align*} 
\end{proof}

\subsection{Some results for weights $w<3q$} \ppar

From now on, we suppose that $w<3q$.

\begin{proposition} \label{prop: proof of Conjecture PBC}
Let $R$ be an $\Fq$-linear relation of weight $w<3q$. Then in $\mathcal S_{\max}(R)$, there exists a tuple of the form $(\fm,(q)+\fs^+,(q-1)+\fn)$ where $\fs$ is a nonempty tuple, $\fm$, $\fn$ are tuples which may be empty.
\end{proposition}

We first show that this proposition implies Theorem \ref{thm: relations from P, B, C}. 
\begin{proof}[Proof of Theorem \ref{thm: relations from P, B, C}]
Proposition \ref{prop: proof of Conjecture PBC} implies that the $\Fq$-vector space of $\Fq$-linear relations of weight $w<3q$ has dimension bounded by the cardinality of the set $\{(\fm,\fs,\fn): \fs \neq \emptyset, w(\fn)+w(\fs)+w(\fn)=w-2q\}$. Prop. \ref{prop: lower bound} shows that the $\Fq$-vector space of $\Fq$-linear relations of weight $w<3q$ has dimension greater than the cardinality of the set $\{(\fm,\fs,\fn): \fs \neq \emptyset, w(\fn)+w(\fs)+w(\fn)=w-2q\}$. Thus we get a proof of Theorem \ref{thm: relations from P, B, C}.
\end{proof}

The rest of this section is devoted to proving this proposition. As we know that there exists an order such that $R < \Ind(R)$, the proof is by induction with respect to this order.

\subsubsection{First case: $\alpha(R)>\alpha(\Ind R)$.} \ppar

\begin{lemma} \label{lemma: proof of Conjecture, 1st case}
Suppose that $w<3q$ and $\alpha(R)>\alpha(\Ind R)$. Then the $\depthlex$-smallest tuple in $\mathcal S_{\max}(R)$ equals $(\fm,(q)+(\{1\}^{k+1})^+,(q-1)+\fn)$ for some $k \geq 0$ and some tuples $\fm$ and $\fn$.
\end{lemma}

\begin{proof}
We consider $\mathcal S_{\max}(R)$ and choose the smallest tuple $\fs$ in this set with respect to the (depth, lex) order. We write $\fs=(\fm,(q)+\fn)$ where $\fm$ and $\fn$ are tuples which may be empty and all the entries of $\fm$ are at most $q$. We also write 
\[ \fs=(\fm',(q-1)+\fn')=(\fm',q-1+n'_1,\dots,n'_{k'})  \] 
such that $n_1' \geq 0$ and $1 \leq n'_2,\dots,n'_{k'} < q-1$. We emphasize that $\fn'$ is not necessarily a tuple as we allow $n_1'=0$. We recall that $\fm$ is a prefix of $\fm'$. Thus we get 
\[ \fs=(\fm,(q)+\frak p,(q-1)+\fn'). \]
More concretly,
\[ \fs=\begin{cases}(\fm,(q)+\frak p,(q-1)+\fn') & \text{if $\mathfrak p \neq \emptyset$,} \\
  (\fm, (q-1) + \fn') & \text{if $\mathfrak{p} = \emptyset$}
  .\end{cases} \]

We consider the special tuple $\ft$ appearing in 
\begin{align*}
\begin{split}
\mathbf{1}_{\{\fm \ne \emptyset\}}\cdot (\fm^{+}, (q-1)*\fn) + (\fm, 1, (q-1)*\fn)
\end{split} 
\end{align*}
defined as follows:
\begin{itemize}
\item[1)] If $\frak p=\emptyset$ which means $\fm=\fm'$ and $(1)+\fn=\fn'$, then 
\[ \ft:=(\fm^+,(q-1)+\fn).\]

\item[2)] If $\frak p \neq \emptyset$ and we can write $\frak p=(\{1\}^k,p_{k+1},\dots,p_\ell)$ with $k \geq 0$ and $p_{k+1} \geq 2$, then we put
\[ \ft:=(\fm^+,\{1\}^k,(q-1)+p_{k+1},\dots,p_\ell,(q-1)+\fn').\]

\item[3)] Otherwise, $\frak p =(\{1\}^k)$ with $k \geq 1$, then we put
\[ \ft:=(\fm^+,\{1\}^k,2(q-1)+\fn').\]
\end{itemize}
Here recall that if $\fm=\emptyset$, then $\fm^+:=(1)$. 

We note that $\alpha(\ft)=\alpha(\fs)=\alpha(R)$. Further, $\depth(\ft)=\depth^+(\fs)$. 

Since $\alpha(R)>\alpha(\Ind R)$ and $\alpha(\ft)=\alpha(R)$, $\ft$ does not appear in $\Ind R$. It follows that $\ft$ also comes from another non basis tuple $\fs_0$ in $R$. Since $\alpha(R)=\alpha(\ft) \leq \alpha(\fs_0) \leq \alpha(R)$, it follows that $\alpha(\fs_0)=\alpha(\ft)=\alpha(R)$, that means $\fs_0 \in \mathcal S_{\max}(R)$. Thus $\fs \lneq \fs_0$ (with respect to the $\depthlex$-order). Now we write two expressions of $\fs_0$:
\[ \fs_0=(\fm_0',(q-1)+\fn_0')=(\fm_0,(q)+\frak p_0,(q-1)+\fn_0'). \]
Here $\fm'_0, \fm_0, \mathfrak p_0$ are tuples. And $\fn'_0$ may not be a tuple and for $i \ge 2$, $1 \leq n'_{0, i}<q-1$. Combining the fact that $\alpha(\fs_0)=\alpha(\ft)=\alpha(R)$ and Lemma \ref{lemma: depth plus} yields:
\begin{itemize}
\item If $\frak p_0=\emptyset$, then 
\begin{itemize}
\item[a)] $\ft=(\fm_0^+,(q-1)+\fn_0)$.
\end{itemize}

\item Otherwise, $\frak p_0 \neq \emptyset$. We write $\frak p_0=(p_{0,1},\dots,p_{0,\ell'})$. Then either there exists $1 \leq i \leq \ell'$ such that 
\begin{itemize}
\item[b)] $\ft=(\fm_0^+,p_{0,1},\dots,p_{0,i},(q-1)+p_{0,i+1},\dots,p_{0,\ell'},(q-1)+\fn_0')$,
\end{itemize}
or 
\begin{itemize}
\item[c)] $\ft=(\fm_0^+,p_{0,1},\dots,p_{0,\ell'},2(q-1)+\fn_0')$. 
\end{itemize}

\end{itemize}

We distinguish three cases:
\begin{itemize}
\item[{\bf Case 1.}] Suppose that $\frak p=\emptyset$. Recall that 
\[ \ft=(\fm^+,(q-1)+\fn).\] 

\noindent {\bf Case 1a.} 
First, suppose that $\mathfrak p_0=\emptyset$. So $\ft=(\fm_0^+,(q-1)+\fn_0)$. Since $\alpha(\fs_0)=\alpha(\fs)=\alpha(\ft)$, it follows that $\depth((q-1)+\fn)=\depth((q-1)+\fn_0)$. Thus $\fm^+=\fm^+_0$ and $\fn=\fn_0$. It implies that $\fs=\fs_0$, which is a contradiction.

\medskip
\noindent {\bf Case 1b.} 
Next, suppose that $\mathfrak p_0=(p_{0,1},\dots,p_{0,\ell'}) \neq \emptyset$ and there exists $1 \leq i \leq \ell'$ such that 
\[ \ft=(\fm_0^+,p_{0,1},\dots,p_{0,i},(q-1)+p_{0,i+1},\dots,p_{0,\ell'},(q-1)+\fn_0').\]
As $\alpha(\fs_0)=\alpha(\fs)=\alpha(\ft)$, it follows that $\depth((q-1)+\fn')=\depth((q-1)+\fn_0')$. It implies that $\fn'=\fn_0'$ and $\fm^+=(\fm_0^+,p_{0,1},\dots,p_{0,i},(q-1)+p_{0,i+1},\dots,p_{0,\ell'})$. Thus $\fm_0 \prec \fm_0^+ \preceq \fm$. It follows that $\fm \neq \emptyset$ which implies $\depth(\fs) = \depth(\fs_0)$. Since $\fm_0 \prec \fm$, we conclude that $\fs=(\fm,(q)+\fn)>\fs_0=(\fm_0,(q)+\frak p_0,(q-1)+\fn_0')$, which is a contradiction. 

\medskip
\noindent {\bf Case 1c.} Finally, we  suppose that $\mathfrak p_0=(p_{0,1},\dots,p_{0,\ell'}) \neq \emptyset$ and
  \[ \ft=(\fm_0^+,\mathfrak{p}_0,2(q-1)+\fn_0').\]
As $\alpha(\ft) = \alpha(\fs) = \alpha(\fs_0)$, we have 
  \[ \fn = (q-1) + \fn'\]
and 
  \[ \fm^+ = (\fm_0^+, \mathfrak{p}_0). \]
It follows that $\fm_0 \prec \fm_0^+ \preceq \fm$. Again, It follows that $\fm \neq \emptyset$ which implies $\depth(\fs) = \depth(\fs_0)$. Since $\fm_0 \prec \fm$, we conclude that $\fs=(\fm,(q)+\fn)>\fs_0=(\fm_0,(q)+\frak p_0,(q-1)+\fn_0')$, which is a contradiction. 

\item[{\bf Case 2.}] If $\frak p \neq \emptyset$ and we can write $\frak p=(\{1\}^k,p_{k+1},\dots,p_\ell)$ with $k \geq 0$ and $p_{k+1} \geq 2$. We recall
\[ \ft=(\fm^+,\{1\}^k,(q-1)+p_{k+1},\dots,p_\ell,(q-1)+\fn').\]

\noindent {\bf Case 2a.} 
First, suppose that $\mathfrak p_0=\emptyset$, $\ft=(\fm_0^+,(q-1)+\fn_0)$. Since $\alpha(\fs_0)=\alpha(\fs)=\alpha(\ft)$, it follows that $\depth((q-1)+\fn')=\depth((q-1)+\fn_0)$. Thus $\fn'=\fn_0$ and
\[ (\fm^+,\{1\}^k,(q-1)+p_{k+1},\dots,p_\ell)=\fm_0^+. \]
Since $(q-1)+p_{k+1} \geq q+1$ and we note that in $\fm_0^+$, all the entries except the last one are at most $q$, it forces $\ell=k+1$ and $p_{k+1}=2$. Thus $(\fm^+,\{1\}^k,q)=\fm_0$ and $\fs=(\fm,(q)+(\{1\}^{k+1})^+,(q-1)+\fn')$. Finally, we note that $\fn'=\fn_0$. Since $\fn_0$ is a tuple, $\fn'$ is also a tuple which may be empty.

\medskip
\noindent {\bf Case 2b.} 
Next, suppose that $\mathfrak p_0 \neq \emptyset$ and we write $\frak p_0=(p_{0,1},\dots,p_{0,\ell'})$. Then there exists $1 \leq i \leq \ell'$ such that 
\[ \ft=(\fm_0^+,p_{0,1},\dots,p_{0,i},(q-1)+p_{0,i+1},\dots,p_{0,\ell'},(q-1)+\fn_0').\]
As $\alpha(\fs_0)=\alpha(\fs)=\alpha(\ft)$, it follows that $\depth((q-1)+\fn')=\depth((q-1)+\fn_0')$. It implies that $\fn'=\fn_0'$ and 
\begin{equation*} 
(\fm^+,\{1\}^k,(q-1)+p_{k+1},\dots,p_\ell)=(\fm_0^+,p_{0,1},\dots,p_{0,i},(q-1)+p_{0,i+1},\dots,p_{0,\ell'}).
\end{equation*}
We denote this tuple  by $\ft'$. If $(q-1)+p_{k+1}$ and $(q-1)+p_{0,i+1}$ are not in the same position, then we get a lower bound for the weight of $\ft'$:
\begin{align*}
w(\ft') &\geq 1+(q-1)+p_{k+1}+(q-1)+p_{0,i+1} \\
& \geq 1+(q+1)+q=2q+2.
\end{align*}
As $\fs=(\ft',(q-1)+\fn')$, it follows that $w(\fs) \geq (2q+2)+(q-1)=3q+1$, which is a contradiction. We conclude that $(q-1)+p_{k+1}$ and $(q-1)+p_{0,i+1}$ are in the same position. Thus 
\begin{equation} \label{eq: Ind=0 eq1}
(\fm^+,\{1\}^k)=(\fm_0^+,p_{0,1},\dots,p_{0,i}).
\end{equation}
and 
\begin{equation*} 
(p_{k+1},\dots,p_\ell)=(p_{0,i+1},\dots,p_{0,\ell'}).
\end{equation*}

If $\fm = \emptyset$, then $\fm_0$ is empty (otherwise $\fm_0^+$ contain an entry $>1$), thus $\fp = \fp_0$, and thus $\fs = \fs_0$, a contradiction.

Therefore, we assume that $\fm \ne \emptyset$. Then we write 
\[ (\fm^+,\{1\}^k)=(m_1, \dots, m_{j-1}, m_{j}+1, \{1\}^k) = (\fm_0^+, p_{0,1}, \dots, p_{0, i-1}). \] 
Since $m_j+1\ge2$, we have $\depth(\fm_0)\le \depth(\fm)$. If $\depth(\fm) = \depth(\fm_0)$, then $\fm = \fm_0$ and $\fp = \fp_0$, which is a contradiction. We conclude that
\[ \depth(\fm_0) < \depth(\fm). \] 
Therefore, $\fm_0 \prec \fm_0^+ \preceq \fm$. It follows that $\fs_0 < \fs$ as $\depth(\fs_0)=\depth(\fs)$ and $\fs_0 \prec \fs$, which is a contradiction.

\medskip 
\noindent {\bf Case 2c.} Now we suppose the following remaining case where $\fp_0=(p_{0, 1}, \dots, p_{0, \ell'}) \ne \emptyset$ and $ \ft = (\fm_0^+, \fp_0 , 2(q-1) + \fn_0')$. Thus
\begin{align*}
  \ft &= (\fm_0^+, \fp_0 , 2(q-1) + \fn_0'),\\
  & = (\fm^+, \{1\}^k, (q-1) + p_{k+1}, \dots, p_{\ell}, (q-1) + \fn').
\end{align*} 
It follows that
\begin{align*}
  (2(q-1) + \fn_0') &= ((q-1) + \fn') \\
  (\fm_0^+, \fp_0) &= (\fm^+, \{1\}^k, (q-1) + p_{k+1}, \dots, p_{\ell}).
\end{align*} 
Since $p_{k+1}\ge2$, we have the inequality 
\begin{align*}
w = w(\ft) & \geq w(\fm^+, \{1\}^k, (q-1) + p_{k+1}, \dots, p_{\ell}) + w(2(q-1) + \fn_0') \\
& \geq 1 + (q-1) + 2 + 2(q-1) = 3q,
\end{align*}
which contradicts to the assumption $w< 3q$. 

\item[{\bf Case 3.}] Otherwise, $\frak p =(\{1\}^k)$ with $k \geq 1$, then we recall
\[ \ft:=(\fm^+,\{1\}^k,2(q-1)+\fn').\]

\noindent {\bf Case 3a.} 
First, suppose that $\mathfrak p_0=\emptyset$, $\ft=(\fm_0^+,(q-1)+\fn_0)$ Since $\alpha(\fs_0)=\alpha(\fs)=\alpha(\ft)$, it follows that $\depth(2(q-1)+\fn')=\depth((q-1)+\fn_0)$. Thus $(q-1)+\fn'=\fn_0$ and
\[ (\fm^+,\{1\}^k)=\fm_0^+. \]
The last entry of $\fm_0^+$ cannot be $1$ unless $\fm_0 = \emptyset$. But when $\fm_0 = \emptyset$, then we find a contradiction again since $1=\depth(\fm_0^+) < k+1=\depth(\fm^+, \{1\}^k)$.

\medskip
\noindent {\bf Case 3b.} 
Next, suppose that $\mathfrak p_0=(p_{0,1},\dots,p_{0,\ell'}) \neq \emptyset$ and there exists $1 \leq i \leq \ell'$ such that 
\[ \ft=(\fm_0^+,p_{0,1},\dots,p_{0,i},(q-1)+p_{0,i+1},\dots,p_{0,\ell'},(q-1)+\fn_0').\]
As $\alpha(\fs_0)=\alpha(\fs)=\alpha(\ft)$, it follows that $\depth((q-1)+\fn')=\depth((q-1)+\fn_0')$. It implies that $(q-1)+\fn'=\fn_0'$ and 
\begin{equation*} 
(\fm^+,\{1\}^k)=(\fm_0^+,p_{0,1},\dots,p_{0,i},(q-1)+p_{0,i+1},\dots,p_{0,\ell'}).
\end{equation*}
Thus $p_{0,\ell'}=1$. Hence the weight of this tuple is at least $1+q+1=q+2$. It follows that $w=w(\ft)=w(\fm^+,\{1\}^k,2(q-1)+\fn')$ is at least $q+2+2(q-1)=3q$, which is a contradiction.

\medskip
\noindent {\bf Case 3c.} Now we suppose the following remaining case of $\fp_0= (p_{0, 1}, \dots, p_{0, \ell'}) \ne \emptyset$ and $\ft = (\fm_0^+, \fp_0 , 2(q-1) + \fn_0')$. Thus
\begin{align*}
  \ft &= (\fm_0^+, \fp_0 , 2(q-1) + \fn_0'),\\
  & = (\fm^+, \{1\}^k, 2(q-1) + \fn')
\end{align*}  
We have $\fn' = \fn'_0$ and 
\[ (\fm_0^+, \fp_0) = (\fm^+, \{1\}^k). \]

If $\fm = \emptyset$, then $\fm_0 = \emptyset$ because otherwise the last entry of $\fm_0^+$ is at least $2$. Therefore, $\fp = \fp_0 = \{1\}^k$ and thus $\fs = \fs_0$, which is a contradiction. 

We assume that $\fm \ne \emptyset$. If $\depth(\fm) < \depth(\fm_0)$;  the last entry of $\fm_0^+$ is $1$ which is contradiction. If $\depth(\fm) = \depth(\fm_0)$, then we conclude that $\fm = \fm_0$ and $\fp = \fp_0$, thus $\fs = \fs_0$ which is also contradiction. 

If $\depth(\fm) > \depth(\fm_0)$, then $\fm_0 \prec \fm_0^+ \prec \fm$. As $\depth(\fs)=\depth(\fs_0$ and $\fs_0=(\fm_0,(q)+\frak p_0,(q-1)+\fn_0') \prec \fs=(\fm,(q)+\{1\}^k,(q-1)+\fn')$, it follows that $\fs_0 < \fs$, which is a contradiction. 
\end{itemize}

To summarize, only Case 2a) can happen and we conclude that $\fs=(\fm,(q)+(\{1\}^{k+1})^+,(q-1)+\fn')$ where $\fm$, $\fn'$ are tuples which may be empty.
\end{proof}

\subsubsection{Second case: $\alpha(R)=\alpha(\Ind R)$.} \ppar

\begin{lemma} \label{lemma: proof of Conjecture, 2nd case}
Suppose that $w<3q$ and $\alpha(R)=\alpha(\Ind R)$. Further, suppose that in $\mathcal S_{\max}(\Ind R)$ there exists a tuple of the form $(\fm_1,(q)+\fs_1^+,(q-1)+\fn_1')$ where $\fs_1 \neq \emptyset$ and $\fm_1$, $\fn_1'$ are tuples.

Then in $\mathcal S_{\max}(R)$ there exists a tuple of the form $(\fm,(q)+\fs^+,(q-1)+\fn')$ where $\fs \neq \emptyset$ and $\fm$, $\fn'$ are tuples.
\end{lemma}

\begin{proof}
By assumption, in $\mathcal S_{\max}(\Ind R)$ there exists a tuple $\ft_1=(\fm_1,(q)+\fs_1^+,(q-1)+\fn_1')$ with $\fs_1 \neq \emptyset$ and $\fm_1$, $\fn_1'$ are tuples. It follows that this tuple comes from some non basis tuple $\fs_0$ in $R$. By Lemma \ref{lemma: bound for alpha}, $\alpha(R)=\alpha(\ft_1) \leq \alpha(\ft_1) \leq\alpha(\fs_0) \leq \alpha(R)$, it follows that $\alpha(\fs_0)=\alpha(\ft_1)=\alpha(R)$. Thus $\fs_0 \in \mathcal S_{\max}(R)$. 

We write two expressions of $\fs_0$:
\[ \fs_0=(\fm_0',(q-1)+\fn_0')=(\fm_0,(q)+\frak p_0,(q-1)+\fn_0'). \]
We distinguish three cases:
\begin{itemize}
\item[a)] First, suppose that $\frak p_0=\emptyset$ and either 
\[ \ft_1=(\fm_0^+,(q-1)+\fn_0) \] 
or 
\[ \ft_1=(\fm_0,1,(q-1)+\fn_0). \] 
As $\alpha(\fs_0)=\alpha(\ft_1)$, it follows that $\fn_1'=\fn_0$. Thus either $(\fm_1,(q)+\fs_1^+)=\fm_0^+$ or $(\fm_1,(q)+\fs_1^+)=(\fm_0,1)$. The latter case cannot happen as all the entries of $(\fm_0,1)$ are at most $q$. It follows that $(\fm_1,(q)+\fs_1^+)=\fm_0^+$. But in $\fm_0^+$, all the entries except the last one are at most $q$. Thus $\depth(\fs_1)=1$, but we get a contradiction by considering the last entry: the last entry of the LHS equals $q+s_1+1 \geq q+2$ but the one on the RHS is at most $q+1$.

\medskip
\item[b)] Second, $\frak p_0=(p_{0,1},\dots,p_{0,\ell'}) \neq \emptyset$. Suppose that there exists $1 \leq i \leq \ell'$ such that either
\[ \ft_1=(\fm_0^+,p_{0,1},\dots,p_{0,i},(q-1)+p_{0,i+1},\dots,p_{0,\ell'},(q-1)+\fn_0'), \]
or 
\[ \ft_1=(\fm_0,1,p_{0,1},\dots,p_{0,i},(q-1)+p_{0,i+1},\dots,p_{0,\ell'},(q-1)+\fn_0'), \]
As $\alpha(\fs_0)=\alpha(\ft_1)$, it follows that $\fn_1'=\fn_0'$. Thus either 
\[ (\fm_1,(q)+\fs_1^+)=(\fm_0^+,p_{0,1},\dots,p_{0,i},(q-1)+p_{0,i+1},\dots,p_{0,\ell'}), \] 
or 
\[ (\fm_1,(q)+\fs_1^+)=(\fm_0,1,p_{0,1},\dots,p_{0,i},(q-1)+p_{0,i+1},\dots,p_{0,\ell'}). \]

Considering the weight yields $w(\fm_0)+w(\frak p_0)+w(\fn_0') \leq q$. Since $\frak p_0 \neq 0$, $w(\fm_0) \leq q-1$. Thus all the entries of $\fm_0^+$ are at most $q$. 

We note that there is at least one entry greater than $q+1$ on the LHS, but on the RHS, there are at most one such entry, i.e., $(q-1)+p_{0,i+1}$. It follows that $p_{0,\ell'}>1$ so that $\frak p_0=(p_{0,1},\dots,p_{0,\ell'}-1)^+$. Thus we obtain $\fs_0 \in \mathcal S_{\max}(R)$ of the desired form $\fs_0=(\fm_0,(q)+(p_{0,1},\dots,p_{0,\ell'}-1)^+,(q-1)+\fn_0')$.

\medskip
\item[c)] Finally, $\frak p_0=(p_{0,1},\dots,p_{0,\ell'}) \neq \emptyset$ and either
\[ \ft_1=(\fm_0^+,p_{0,1},\dots,p_{0,\ell'},2(q-1)+\fn_0'), \]
or 
\[ \ft_1=(\fm_0,1,p_{0,1},p_{0,\ell'},2(q-1)+\fn_0'), \]
As $\alpha(\fs_0)=\alpha(\ft_1)$, it follows that $\fn_1'=(q-1)+\fn_0'$. Thus either 
\[ (\fm_1,(q)+\fs_1^+)=(\fm_0^+,p_{0,1},\dots,p_{0,\ell'}), \] 
or 
\[ (\fm_1,(q)+\fs_1^+)=(\fm_0,1,p_{0,1},\dots,p_{0,\ell'}). \]

Considering the weight yields 
\begin{align*}
w &=w(\ft_1) \\
&\geq w(\fm_1,(q)+\fs_1^+) + w(2(q-1)+\fn_0') \\
&\geq (q+2)+(2q-2)=3q.
\end{align*}
We obtain a contradiction in this case.
\end{itemize}
\end{proof}

\subsubsection{Proof of Prop. \ref{prop: proof of Conjecture PBC}} \ppar

Let $R$ be an $\Fq$-linear relation of weight $w<3q$. Then there exists $k \geq 1$ such that $\Ind^k(R) \neq 0$ but $\Ind^{k+1}(R)=0$.

We prove by descending induction on $i$ that for all $0 \leq i \leq k$, in $\mathcal S_{\max}(\Ind^i(R))$, there exists a tuple of the form $(\fm,(q)+\fs^+,(q-1)+\fn)$ where $\fs$ is a nonempty tuple, $\fm$, $\fn$ are tuples which may be empty.

For $i=k$, as $\Ind^k(R) \neq 0$ and $\Ind^{k+1}(R)=0$, $\alpha(\Ind^k(R))>\alpha(\Ind^{k+1}(R))$. We are done by Lemma \ref{lemma: proof of Conjecture, 1st case}.

Suppose that the claim holds for $1 \leq i+1 \leq k$. We show that the claim also holds for $i$. If $\alpha(\Ind^i(R))>\alpha(\Ind^{i+1}(R))$, it is done by Lemma \ref{lemma: proof of Conjecture, 1st case}. Otherwise, $\alpha(\Ind^i(R))=\alpha(\Ind^{i+1}(R))$, it is done by Lemma \ref{lemma: proof of Conjecture, 2nd case}.

Finally, we take $i=0$ which implies that  in $\mathcal S_{\max}(R)$, there exists a tuple of the form $(\fm,(q)+\fs^+,(q-1)+\fn)$ where $\fs$ is a nonempty tuple, $\fm$, $\fn$ are tuples which may be empty. We are done.


\end{document}